\documentclass[a4paper,fleqn]{cas-sc}

\usepackage[authoryear,longnamesfirst]{natbib}

\def\tsc#1{\csdef{#1}{\textsc{\lowercase{#1}}\xspace}}
\tsc{WGM}
\tsc{QE}

\begin{document}
\let\WriteBookmarks\relax
\def\floatpagepagefraction{1}
\def\textpagefraction{.001}

\shorttitle{Generalization of Hallaire-Luikov
Moisture Transfer Equation}    

\shortauthors{Karimov E., Khasanov Sh.}  

\title [mode = title]{Generalization of Hallaire-Luikov  Moisture Transfer Equation: Direct Problem with the \texorpdfstring{$\psi$}{psi}-Prabhakar Operator}  

\tnotemark[1] 

\tnotetext[1]{} 

%

\author[1]{Erkinjon Karimov}

\cormark[1]

\fnmark[1]

\ead{erkinjon.karimov@ugent.be}

\ead[url]{}

\credit{Conceptualization of this study, Methodology, Revision}

\affiliation[1]{organization={Ghent University},
            addressline={Krijgslaan str.297}, 
            city={Ghent},
            postcode={9000}, 
            state={East Flanders},
            country={Uzbekistan}}

\author[2]{Shokhzodbek Khasanov}

\fnmark[2]

\ead{shohzodbekxasanov@gmail.com}

\ead[url]{}

\credit{Methodology, Revision}

\affiliation[2]{organization={Fergana state University},
            addressline={Murabbiylar str. 19}, 
            city={Fergana},
            postcode={150100}, 
            state={Fergana region},
            country={Uzbekistan}}

\cortext[1]{Corresponding author}

\fntext[1]{}


\begin{abstract}
This paper focuses on the analysis of an initial-boundary value (direct) problem 
for the Hallaire-Luikov moisture transfer equation involving the 
$\psi$-Prabhakar integral-differential operator of fractional order. 
We establish the existence, uniqueness, and stability of the solution 
to the formulated problem. To construct the solution, we employ the 
method of separation of variables and the method of successive 
approximations (iteration method), and obtain the solution to the 
considered problem in an explicit form. Furthermore, the solution is 
expressed in terms of a novel quadrivariate Mittag-Leffler-type 
function. An a priori estimate for the problem is also established.
\end{abstract}





\maketitle

\section{Introduction}\label{intro}
Moisture transport in porous and capillary-porous media plays a crucial role in a wide range of natural and engineering processes, including soil drying and infiltration, building physics, geothermal systems, and heat-mass transfer in heterogeneous materials. Classical models of moisture transfer are typically based on diffusion-type equations derived under the assumption of local equilibrium and Fickian transport. One of the earliest and most widely used models is the nonlinear diffusion equation proposed in soil physics, in which the moisture flux is driven solely by gradients of capillary pressure \cite{JCH02}.

However, experimental observations reveal that moisture transport in real porous materials often exhibits non-Fickian behavior, such as inverse moisture flows, delayed response, and memory effects caused by complex pore structures, heterogeneity, and multiscale interactions. To address these phenomena, Hallaire introduced a modified moisture transfer equation incorporating a mixed space-time derivative term, which allows the description of moisture migration from drier to wetter regions \cite{H63}. Subsequently, Luikov proposed a further modification accounting for finite propagation speed, leading to the well-known Hallaire-Luikov moisture transfer equation \cite{L65}. These models have been studied extensively and successfully applied in various contexts of heat and mass transfer in porous media.

Despite their success, classical integer-order models remain limited in their ability to capture anomalous diffusion and long-range memory effects inherent in many real materials. This limitation has motivated the introduction of fractional-order generalizations, where the classical time derivative is replaced by a fractional operator. Fractional derivatives naturally incorporate memory and hereditary properties and have proven effective in modeling transport processes in fractal or heterogeneous media. In this direction, several authors have investigated fractional versions of the Hallaire-Luikov equation using Riemann-Liouville \cite{KG19}, Caputo \cite{IMS25}, Hilfer \cite{AKK23}, and related fractional derivatives, establishing well-posedness results and, in some cases, explicit solution representations.

In recent years, increasing attention has been paid to generalized fractional operators with multi-parameter kernels, which provide additional flexibility in modeling relaxation and memory effects beyond simple power-law behavior. Among these, the Prabhakar fractional operator, based on the generalized Mittag–Leffler function, has emerged as a powerful tool for describing multi-scale and non-exponential memory processes \cite{FRS22}. An important further extension is the $\psi$‑Prabhakar fractional operator, where the fractional differentiation is taken with respect to a monotone function $\psi$. This generalization allows the incorporation of non-uniform time scaling and variable temporal metrics, making it particularly suitable for modeling transport processes in heterogeneous or evolving media \cite{O22}.

Motivated by these developments, the present paper is devoted to the analysis of a fractional generalization of the Hallaire-Luikov moisture transfer equation involving the $\psi$-Prabhakar integral-differential operator. The use of the 
$\psi$-Prabhakar operator enables a unified framework that encompasses several known fractional models as special cases, while also allowing a more accurate description of complex memory effects.

The main objective of this work is to study a direct problem for the proposed fractional model in a bounded spatial domain. We establish rigorous results concerning the existence, uniqueness, and stability of solutions under suitable assumptions on the data and parameters. The analysis is based on the eigenfunction expansions and energy-type estimates adapted to the $\psi$-Prabhakar framework.

A notable feature of this study is the derivation of an explicit solution representation in terms of newly introduced multivariate Mittag-Leffler-type functions, including trivariate and quadri-variate generalizations. 

\section{Problem formulation and formal solution}\label{probform}

\subsection{Formulation of a problem and apriori estimate}

Let us consider the following generalized Hallaire-Luikov moisture transfer equation:
\begin{equation}
    a\,\,{}^PD_{0t;\psi }^{\alpha ,\beta  + 1,\gamma ,\delta }u(t,x) + {}^PD_{0t;\psi }^{\alpha ,\beta ,\gamma ,\delta }u(t,x) = {u_{xx}}(t,x) + b\,\,{}^PD_{0t;\psi }^{\alpha ,\theta ,\gamma ,\delta }{u_{xx}}(t,x) + f\left( {t,x} \right),
\label{eq:31}
\end{equation}
in the following domain\,${{\Omega }_{\left( T,1 \right)}}=\left\{ \left( t,x \right):0<t<T,0<x<1\,\right\}$. Here, 
\begin{equation}\label{psiPD}
    {}^PD_{0t;\psi }^{\alpha ,\beta ,\gamma ,\delta }f(t) = {\left( {\frac{1}{{\psi '(t)}}\frac{d}{{dt}}} \right)^n}{}^PI_{0t;\psi }^{\alpha ,n - \beta ,\, - \gamma ,\delta }f(t)
\end{equation}
is the $\psi$-Prabhakar fractional derivative,
\begin{equation}\label{psiPI}
{}^PI_{0t;\psi }^{\alpha ,\beta ,\,\gamma ,\delta }f(t) = \int\limits_0^t {\psi '(s){{\left( {\psi (t) - \psi (s)} \right)}^{\beta  - 1}}E_{\alpha ,\beta }^{  \gamma }\left( {\delta {{\left( {\psi (t) - \psi (s)} \right)}^\alpha }} \right)f\left( s \right)ds,\,\,} t > 0
\end{equation}
is $\psi$-Prabhakar fractional integral, $E_{\alpha ,\beta }^{\gamma }(z) :=\sum\limits_{k=0}^{+\infty }{\frac{{{\left( \gamma  \right)}_{k}}}{\Gamma \left( \alpha k+\beta  \right){k!}}}\,{{z}^{k}}$ represents the Prabhakar function, $n = \left[ \beta  \right] + 1,n \in \mathbb{N},$\,$a,b,T = const > 0,$\,$\alpha ,\beta ,\theta ,\gamma ,\delta \in \mathbb{R},\,\,\alpha >0,\,0<\theta<\beta <1.$  Also, $f\left( t,x \right)$ is a given function, $\psi \left( t \right)\in C^{1}\,[0,T]$ is an increasing function that ${\psi }'\left( t \right)>0$ for all $t\in [0.T]$. \\ 

\textbf{Problem.} 
Find a regular solution $u(t,x)$ of equation \eqref{eq:31} that satisfies 
\begin{itemize}
\item regularity conditions:
${\left( {\psi \left( t \right) - \psi \left( 0 \right)} \right)^{1 - \beta }}u\left( {t,x} \right) \in C\left( {{{\bar \Omega }_{\left( {T,1} \right)}}} \right)\,,\,{}^PD_{0t;\psi }^{\alpha ,\beta  + 1,\gamma ,\delta }u(t,x) \in C\left( {{\Omega _{\left( {T,1} \right)}}} \right),$\\
${}^PD_{0t;\psi }^{\alpha ,\beta ,\gamma ,\delta }u(t,x) \in C\left( {{\Omega _{\left( {T,1} \right)}}} \right),\,{u_{xx}}(t,x) \in C\left( {{\Omega _{\left( {T,1} \right)}}} \right)\,,{}^PD_{0t;\psi }^{\alpha ,\theta ,\gamma ,\delta }{u_{xx}}(t,x) \in C\left( {{\Omega _{\left( {T,1} \right)}}} \right);$
\item initial conditions:
$\mathop {\lim }\limits_{t \to  + 0} {}^PI_{0t;\psi }^{\alpha ,1 - \beta , - \gamma ,\delta }u(t,x) = \omega (x),\,\,\,\,\mathop {\lim }\limits_{t \to  + 0} {}^PD_{0t;\psi }^{\alpha ,\beta ,\gamma ,\delta }u(t,x) = \varphi (x)\,,\,\,x \in \left[ {0,1} \right]\,;$
\item boundary conditions:
$u(t,0) = 0,\,\,u(t,1) = 0\,,\,\,\,t \in \left( {0,T} \right]\,.$
\end{itemize}
Here, $\varphi (x)\,,\,\,\omega (x)\,$ are real given functions such that $\varphi (0)=\omega (0)=0\,$.

{\bf Physical meaning of components of Eq.\eqref{eq:31}.} 
The operator ${}^PD_{0t;\psi}^{\alpha,\beta,\gamma,\delta}u(t,x)$ is a generalized fractional derivative with a Prabhakar kernel and time-scaling function $\psi(t)$ introduces long-term memory into the moisture balance equation. Models delayed moisture accumulation and release due to pore-scale trapping, sorption-desorption hysteresis,
internal material heterogeneity. The function $\psi(t)$ allows modeling nonuniform temporal processes, such as aging or temperature-dependent transport. This term also represents generalized moisture accumulation in the porous medium. Unlike classical storage terms, it accounts for the dependence of the current moisture state on its entire temporal history. ${}^P D_{0t;\psi}^{\alpha,\beta+1,\gamma,\delta} u(t,x)$
corresponds to moisture relaxation or inertia effects. The coefficient $a$ characterizes the resistance of the medium to rapid changes in moisture content and reflects internal structural constraints. $u_{xx}(t,x)$ is the classical Fickian diffusion term, representing moisture flow driven by spatial gradients. $b\,{}^P D_{0t;\psi}^{\alpha,\theta,\gamma,\delta} u_{xx}(t,x)$
introduces memory-dependent diffusion, indicating that the moisture flux depends on past gradients. The coefficient $b$ quantifies the strength of this nonlocal transport mechanism and is associated with time-dependent permeability and pore connectivity. $f(t,x)$ models external or internal sources and sinks of moisture, such as evaporation, condensation, rainfall, or imposed boundary fluxes. 

The following statement is related with the apriori estimate:
\newtheorem{theorem}{Theorem}
\begin{theorem}
    Let $f(t,x)\in C\left( {{{\bar{\Omega }}}_{\left( T,1 \right)}} \right),\varphi \left( x \right)\in C\left[ 0,1 \right],\omega \left( x \right)\in {{C}^{2}}\left[ 0,1 \right],\omega \left( 0 \right)=\omega \left( 1 \right)=0$ and $\psi \left( t \right)\in C^{1}\,[0,T],\,$${\psi }'\left( t \right)>0$ for all $t\in [0.T]$. Then the following apriori estimate holds for the solution of the problem:
    $$
    \frac{a}{2}\left\| {{}^PD_{0t;\psi }^{\alpha ,\beta ,\gamma ,\delta }u(t,x)} \right\|_{{L_2}(0,1)}^2 + \left( {1 - {b_1}} \right)\int\limits_0^t {\psi '\left( s \right)\left\| {{}^PD_{0s;\psi }^{\alpha ,\beta ,\gamma ,\delta }u(s,x)} \right\|_{{L_2}(0,1)}^2ds}- 
    $$
    $$
     2{\varepsilon _1}\int\limits_0^t {\psi '\left( s \right)\left\| {{u_{xx}}(s,x)} \right\|_{{L_2}(0,1)}^2\,} ds - 2b{\varepsilon _2}\int\limits_0^t {\psi '\left( s \right)\left\| {{}^PD_{0s;\psi }^{\alpha ,\theta ,\gamma ,\delta }{u_{xx}}(s,x)} \right\|_{{L_2}(0,1)}^2ds}\le  
    $$
    $$
    2\varepsilon \int\limits_0^t {\psi '\left( s \right)\left\| {f\left( {s,x} \right)} \right\|_{{L_2}(0,1)}^2ds}  + T\left\| {\omega ''\left( x \right)} \right\|_{{L_2}(0,1)}^2 + \frac{a}{2}\left\| {\varphi \left( x \right)} \right\|_{{L_2}(0,1)}^2\,,
    $$
    where $\varepsilon,\varepsilon_{1},\varepsilon_{2},b_{1},T$ are positive real numbers.
\end{theorem}
\textbf{Proof:}
To prove this statement we introduce new function as
 $$
    \vartheta(t,x) = u(t,x) - {\left( {\psi \left( t \right) - \psi \left( 0 \right)} \right)^{\beta  - 1}}E_{\alpha,\beta }^\gamma \left( {\delta {{\left( {\psi \left( t \right) - \psi \left( 0 \right)} \right)}^\alpha }} \right)\,\omega \left( x \right)\,.
    $$
Then, with respect to the new function, 
and including the following formulas in \cite{P71}  

    $$
     E_{\alpha ,\beta }^0\left( z \right) = \frac{1}{{\Gamma \left( \beta  \right)}},\,\,\,\,\,\,
     E_{\alpha ,0}^0\left( z \right) = 0,
    $$
and the relation in \cite{O21} 
$$
 \begin{array}{l}
{}^PD_{0t;\psi }^{\alpha ,\beta ,\gamma ,\delta }\left( {{{\left( {\psi \left( t \right) - \psi \left( 0 \right)} \right)}^{\nu  - 1}}E_{\alpha ,\nu }^\sigma \left( {\delta {{\left( {\psi \left( t \right) - \psi \left( 0 \right)} \right)}^\alpha }} \right)} \right)\\
 = {\left( {\psi \left( t \right) - \psi \left( 0 \right)} \right)^{\nu  - \beta  - 1}}E_{\alpha ,\nu  - \beta }^{\sigma  - \gamma }\left( {\delta {{\left( {\psi \left( t \right) - \psi \left( 0 \right)} \right)}^\alpha }} \right)\,,
\end{array}
$$
with $\alpha,\beta,\gamma,\delta,\nu,\sigma \in\mathbb{R}$,\,$\alpha  > 0,\beta  \ge 0,\nu  > 0,$ we obtain the following relations 
$$
{u_{xx}}(t,x) = {\vartheta _{xx}}(t,x) + {\left( {\psi \left( t \right) - \psi \left( 0 \right)} \right)^{\beta  - 1}}E_{\alpha ,\beta }^\gamma \left( {\delta {{\left( {\psi \left( t \right) - \psi \left( 0 \right)} \right)}^\alpha }} \right)\,\omega ''\left( x \right);
$$
$$
{}^PD_{0t;\psi }^{\alpha ,\beta ,\gamma ,\delta }u(t,x) = {}^PD_{0t;\psi }^{\alpha ,\beta ,\gamma ,\delta }\vartheta (t,x)\,,
{}^PD_{0t;\psi }^{\alpha ,\beta  + 1,\gamma ,\delta }u(t,x)= {}^PD_{0t;\psi }^{\alpha ,\beta  + 1,\gamma ,\delta }\vartheta (t,x)\,;
$$
$$
{}^PD_{0t;\psi }^{\alpha ,\theta ,\gamma ,\delta }{u_{xx}}(t,x) = {}^PD_{0t;\psi }^{\alpha ,\theta ,\gamma ,\delta }{\vartheta _{xx}}(t,x) + \frac{{{{\left( {\psi \left( t \right) - \psi \left( 0 \right)} \right)}^{\beta  - \theta  - 1}}}}{{\Gamma \left( {\beta  - \theta } \right)}}\,\omega ''\left( x \right)\,;
$$
$$
\mathop {\lim }\limits_{t \to  + 0} {}^PI_{0t;\psi }^{\alpha ,1 - \beta ,-\gamma ,\delta }\vartheta (t,x) = 0,\,\,
\mathop {\lim }\limits_{t \to  + 0} {}^PD_{0t;\psi }^{\alpha ,\beta ,\gamma ,\delta }\vartheta (t,x) = \varphi (x),\,\,
\vartheta (t,0) = \,\vartheta (t,1) = 0.
$$
Since then, we obtain with respect to the new function the following equation    
\begin{equation}
    a\,\,\,{}^PD_{0t;\psi }^{\alpha ,\beta  + 1,\gamma ,\delta }\vartheta (t,x) + {}^PD_{0t;\psi }^{\alpha ,\beta ,\gamma ,\delta }\vartheta (t,x) = {\vartheta _{xx}}(t,x) + \,b\,\,\,{}^PD_{0t;\psi }^{\alpha ,\theta ,\gamma ,\delta }{\vartheta _{xx}}(t,x) + g(t,x),\,\,(t,x) \in {\Omega _{\,\left( {T,1} \right)}}\,,
    \label{eq:32}
\end{equation}
where
$$
g(t,x) = f\left( {t,x} \right) + {\left( {\psi \left( t \right) - \psi \left( 0 \right)} \right)^{\beta  - 1}}E_{\alpha ,\beta }^\gamma \left( {\delta {{\left( {\psi \left( t \right) - \psi \left( 0 \right)} \right)}^\alpha }} \right)\,\omega ''\left( x \right) + \frac{{b\,\omega ''\left( x \right)}}{{{{\left( {\psi \left( t \right) - \psi \left( 0 \right)} \right)}^{1 - \left( {\beta  - \theta } \right)}}\Gamma \left( {\beta  - \theta } \right)}}\,,
$$
and initial conditions
\begin{equation}
\mathop {\lim }\limits_{t \to  + 0} {}^PI_{0t;\psi }^{\alpha ,1 - \beta , - \gamma ,\delta }\vartheta (t,x) = 0,\,\,\,\,\mathop {\lim }\limits_{t \to  + 0} {}^PD_{0t;\psi }^{\alpha ,\beta ,\gamma ,\delta }\vartheta (t,x) = \varphi (x)\,,\,\,x \in \left[ {0,1} \right]\,,
    \label{eq:34}
\end{equation}
boundary conditions
\begin{equation}
\vartheta (t,0) = 0,\,\,\vartheta (t,1) = 0\,,\,\,\,t \in \left( {0,T} \right]\,.
    \label{eq:35}
\end{equation}  
If we multiply both sides of equation~\eqref{eq:32} by 
${}^P D_{0t;\psi}^{\alpha,\beta,\gamma,\delta}\vartheta(t,x)$
and integrate the resulting expression with respect to $x$ from $0$ to $1$, then we obtain:
\begin{equation}
    \begin{array}{l}
a\,\left( {{}^PD_{0t;\psi }^{\alpha ,\beta  + 1,\gamma ,\delta }\vartheta (t,x),{}^PD_{0t;\psi }^{\alpha ,\beta ,\gamma ,\delta }\vartheta (t,x)} \right) + \left( {{}^PD_{0t;\psi }^{\alpha ,\beta ,\gamma ,\delta }\vartheta (t,x),{}^PD_{0t;\psi }^{\alpha ,\beta ,\gamma ,\delta }\vartheta (t,x)} \right) = \left( {{\vartheta _{xx}}(t,x),{}^PD_{0t;\psi }^{\alpha ,\beta ,\gamma ,\delta }\vartheta (t,x)} \right) \\
+ b\left( {{}^PD_{0t;\psi }^{\alpha ,\theta ,\gamma ,\delta }{\vartheta _{xx}}(t,x),{}^PD_{0t;\psi }^{\alpha ,\beta ,\gamma ,\delta }\vartheta (t,x)} \right) + \left( {f\left( {t,x} \right),{}^PD_{0t;\psi }^{\alpha ,\beta ,\gamma ,\delta }\vartheta (t,x)} \right)\,.
 \label{eq:36}
\end{array}
\end{equation}
We simplify each term of equation~\eqref{eq:36}, 
considering the boundary condition~\eqref{eq:35}, Cauchy-Bunyakovsky and $\varepsilon -$ Young inequalities in \cite{Samarskiy89}
$$
 \left| {\left( {u,v} \right)} \right| \le \left\| u \right\|_{{L_2}(0,1)}\left\| v \right\|_{{L_2}(0,1)}  \le \varepsilon {\left\| u \right\|_{{L_2}(0,1)}^2} + \frac{1}{{4\varepsilon }}{\left\| v \right\|_{{L_2}(0,1)}^2},
$$
and the following notations:
$$
\left( {u,v} \right) = \int\limits_0^1 {uvdx} ,\,\,\,\,\left( {u,u} \right) = \left\| u \right\|_{{L_2}(0,1)}^2\,.
$$
Namely,
$$
\begin{array}{l}
a\,\left( {{}^PD_{0t;\psi }^{\alpha ,\beta  + 1,\gamma ,\delta }\vartheta (t,x),{}^PD_{0t;\psi }^{\alpha ,\beta ,\gamma ,\delta }\vartheta (t,x)} \right)\\
 = a\int\limits_0^1 {\frac{1}{{\Gamma \left( {1 - \beta } \right)}}{{\left( {\frac{1}{{\psi '\left( t \right)}}\frac{d}{{dt}}} \right)}^2}\int\limits_0^t {\psi '\left( s \right){{\left( {\psi \left( t \right) - \psi \left( s \right)} \right)}^{ - \beta }}E_{\alpha ,1 - \beta }^{ - \gamma }} \left( {\delta {{\left( {\psi \left( t \right) - \psi \left( s \right)} \right)}^\alpha }} \right)\vartheta (s,x)ds} \\
 \times \frac{1}{{\Gamma \left( {1 - \beta } \right)}}\left( {\frac{1}{{\psi '\left( t \right)}}\frac{d}{{dt}}} \right)\int\limits_0^t {\psi '\left( s \right){{\left( {\psi \left( t \right) - \psi \left( s \right)} \right)}^{ - \beta }}E_{\alpha ,1 - \beta }^{ - \gamma }} \left( {\delta {{\left( {\psi \left( t \right) - \psi \left( s \right)} \right)}^\alpha }} \right)\vartheta (s,x)ds\,dx\\
 = \frac{a}{2}\int\limits_0^1 {\left( {\frac{1}{{\psi '\left( t \right)}}\frac{d}{{dt}}} \right){{\left( {{}^PD_{0t;\psi }^{\alpha ,\beta ,\gamma ,\delta }\vartheta (t,x)} \right)}^2}dx}  = \frac{a}{2}\left( {\frac{1}{{\psi '\left( t \right)}}\frac{d}{{dt}}} \right)\left\| {{}^PD_{0t;\psi }^{\alpha ,\beta ,\gamma ,\delta }\vartheta (t,x)} \right\|_{{L_2}(0,1)}^2,
\end{array}
$$
$$
\left( {{}^PD_{0t;\psi }^{\alpha ,\beta ,\gamma ,\delta }\vartheta (t,x),{}^PD_{0t;\psi }^{\alpha ,\beta ,\gamma ,\delta }\vartheta (t,x)} \right) = \left\| {{}^PD_{0t;\psi }^{\alpha ,\beta ,\gamma ,\delta }\vartheta (t,x)} \right\|_{{L_2}(0,1)}^2,
$$
$$
\left( {g\left( {t,x} \right),{}^PD_{0t;\psi }^{\alpha ,\beta ,\gamma ,\delta }\vartheta (t,x)} \right) \le \,\,\,{{\varepsilon }}\left\| {g\left( {t,x} \right)} \right\|_{{L_2}(0,1)}^2 + \frac{1}{{4\varepsilon }} \left\| {{}^PD_{0t;\psi }^{\alpha ,\beta ,\gamma ,\delta }\vartheta (t,x)} \right\|_{{L_2}(0,1)}^2\,\,,
$$
$$
    \left( {{\vartheta _{xx}}(t,x),{}^PD_{0t;\psi }^{\alpha ,\beta ,\gamma ,\delta }\vartheta (t,x)} \right) \le {\varepsilon _1}\left\| {{\vartheta _{xx}}(t,x)} \right\|_{{L_2}(0,1)}^2 + \frac{1}{{4{\varepsilon _1}}}\left\| {{}^PD_{0t;\psi }^{\alpha ,\beta ,\gamma ,\delta }\vartheta (t,x)} \right\|_{{L_2}(0,1)}^2\,,
$$
$$
b\left( {{}^PD_{0t;\psi }^{\alpha ,\theta ,\gamma ,\delta }{\vartheta _{xx}}(t,x),{}^PD_{0t;\psi }^{\alpha ,\beta ,\gamma ,\delta }\vartheta (t,x)} \right) \le b{\varepsilon _2}\left\| {{}^PD_{0t;\psi }^{\alpha ,\theta ,\gamma ,\delta }{\vartheta _{xx}}(t,x)} \right\|_{{L_2}(0,1)}^2 + \frac{b}{{4{\varepsilon _2}}}\left\| {{}^PD_{0t;\psi }^{\alpha ,\beta ,\gamma ,\delta }\vartheta (t,x)} \right\|_{{L_2}(0,1)}^2\,.
$$
Consequently, we reduce equation \eqref{eq:36} to the following inequality:

\begin{equation}
\begin{array}{l}
\frac{a}{2}\left( {\frac{1}{{\psi '\left( t \right)}}\frac{d}{{dt}}} \right)\left\| {{}^PD_{0t;\psi }^{\alpha ,\beta ,\gamma ,\delta }\vartheta (t,x)} \right\|_{{L_2}(0,1)}^2 + \left( {1 - \frac{1}{{4{\varepsilon _1}}} - \frac{b}{{4{\varepsilon _2}}} - \frac{1}{{4\varepsilon }}} \right)\left\| {{}^PD_{0t;\psi }^{\alpha ,\beta ,\gamma ,\delta }\vartheta (t,x)} \right\|_{{L_2}(0,1)}^2 \\
\le {\varepsilon _1}\left\| {{\vartheta _{xx}}(t,x)} \right\|_{{L_2}(0,1)}^2 + b{\varepsilon _2}\left\| {{}^PD_{0t;\psi }^{\alpha ,\theta ,\gamma ,\delta }{\vartheta _{xx}}(t,x)} \right\|_{{L_2}(0,1)}^2 + \varepsilon \left\| {g\left( {t,x} \right)} \right\|_{{L_2}(0,1)}^2.
\end{array}
    \label{eq:37}
\end{equation}

If we apply integral operator ${}^PI_{0t;\psi }^{\alpha ,1,0,\delta }$ to both side of inequality \eqref{eq:37}, then
$$
\begin{array}{l}
\frac{a}{2}{}^PI_{0t;\psi }^{\alpha ,1,0,\delta }\left( {\frac{1}{{\psi '\left( t \right)}}\frac{d}{{dt}}} \right)\left\| {{}^PD_{0t;\psi }^{\alpha ,\beta ,\gamma ,\delta }\vartheta (t,x)} \right\|_{{L_2}(0,1)}^2 + \left( {1 - \frac{1}{{4{\varepsilon _1}}} - \frac{b}{{4{\varepsilon _2}}} - \frac{1}{{4\varepsilon }}} \right){}^PI_{0t;\psi }^{\alpha ,1,0,\delta }\left\| {{}^PD_{0t;\psi }^{\alpha ,\beta ,\gamma ,\delta }\vartheta (t,x)} \right\|_{{L_2}(0,1)}^2 \\
\le {\varepsilon _1}{}^PI_{0t;\psi }^{\alpha ,1,0,\delta }\left\| {{\vartheta _{xx}}(t,x)} \right\|_{{L_2}(0,1)}^2 + b{\varepsilon _2}{}^PI_{0t;\psi }^{\alpha ,1,0,\delta }\left\| {{}^PD_{0t;\psi }^{\alpha ,\theta ,\gamma ,\delta }{\vartheta _{xx}}(t,x)} \right\|_{{L_2}(0,1)}^2 + \varepsilon {}^PI_{0t;\psi }^{\alpha ,1,0,\delta }\left\| {g\left( {t,x} \right)} \right\|_{{L_2}(0,1)}^2.
\end{array}
$$
In the particular case of the formula in \cite{O21}
\begin{equation}
        {}^PI_{0t;\psi }^{\alpha ,\beta ,\,\gamma ,\delta }\,{}^PD_{0t;\psi }^{\alpha ,\beta ,\gamma ,\delta }f(t) = f\left( t \right) - \sum\limits_{k = 1}^n {{A_k}{{\left( {\psi \left( t \right) - \psi \left( 0 \right)} \right)}^{\beta  - k}}E_{\alpha ,\beta  + 1 - k}^\gamma \left( {\delta {{\left( {\psi \left( t \right) - \psi \left( 0 \right)} \right)}^\alpha }} \right)} ,
    \label{eq:30}
    \end{equation}
where ${A_k} = \mathop {\lim }\limits_{t \to  + a} {}^PD_{0t;\psi }^{\alpha ,\beta  - k,\gamma ,\delta }f(t)\,$ and $\alpha,\beta,\gamma,\delta \in \mathbb{R}$,\, $f\in  C^n( \Omega.$\,\,\, For$f\in  C^{1}( \Omega$,we obtain the following equality
$$
{}^PI_{at;\psi }^{\alpha ,1,0,\delta }\left( {\frac{1}{{\psi '\left( t \right)}}\frac{d}{{dt}}} \right)f(t) = f(t) - f(a),\,\,\,\,f(a) = \mathop {\lim }\limits_{t \to  + a} f(t)\,.
$$
Using the last equality and second condition \eqref{eq:34}, we get
\begin{equation}
\begin{array}{l}
\frac{a}{2}\left\| {{}^PD_{0t;\psi }^{\alpha ,\beta ,\gamma ,\delta }\vartheta (t,x)} \right\|_{{L_2}(0,1)}^2 + \left( {1 - \frac{1}{{4{\varepsilon _1}}} - \frac{b}{{4{\varepsilon _2}}} - \frac{1}{{4\varepsilon }}} \right)\int\limits_0^t {\psi '\left( s \right)\left\| {{}^PD_{0s;\psi }^{\alpha ,\beta ,\gamma ,\delta }\vartheta (s,x)} \right\|_{{L_2}(0,1)}^2ds} \\
\le {\varepsilon _1}\int\limits_0^t {\psi '\left( s \right)\left\| {{\vartheta _{xx}}(s,x)} \right\|_{{L_2}(0,1)}^2} ds + b{\varepsilon _2}\int\limits_0^t {\psi '\left( s \right)\left\| {{}^PD_{0s;\psi }^{\alpha ,\theta ,\gamma ,\delta }{\vartheta _{xx}}(s,x)} \right\|_{{L_2}(0,1)}^2ds}  \\
+\varepsilon \int\limits_0^t {\psi '\left( s \right)\left\| {g\left( {s,x} \right)} \right\|_{{L_2}(0,1)}^2ds}  + \frac{a}{2}\left\| {\varphi \left( x \right)} \right\|_{{L_2}(0,1)}^2.
\end{array}
    \label{eq:40}
\end{equation}
Applying suitable substitutions for $\vartheta(t,x)$ and $g(t,x)$ in inequality~\eqref{eq:40},  
and performing the necessary calculations, we arrive at the following inequality:
$$
\begin{array}{l}
\frac{a}{2}\left\| {{}^PD_{0t;\psi }^{\alpha ,\beta ,\gamma ,\delta }u(t,x)} \right\|_{{L_2}(0,1)}^2 + \left( {1 - {b_1}} \right)\int\limits_0^t {\psi '\left( s \right)\left\| {{}^PD_{0s;\psi }^{\alpha ,\beta ,\gamma ,\delta }u(s,x)} \right\|_{{L_2}(0,1)}^2ds} \\
 - 2{\varepsilon _1}\int\limits_0^t {\psi '\left( s \right)\left\| {{u_{xx}}(s,x)} \right\|_{{L_2}(0,1)}^2\,} ds - 2b{\varepsilon _2}\int\limits_0^t {\psi '\left( s \right)\left\| {{}^PD_{0s;\psi }^{\alpha ,\theta ,\gamma ,\delta }{u_{xx}}(s,x)} \right\|_{{L_2}(0,1)}^2ds} \\
\le 2\varepsilon \int\limits_0^t {\psi '\left( s \right)\left\| {f\left( {s,x} \right)} \right\|_{{L_2}(0,1)}^2ds}  + T\left\| {\omega ''\left( x \right)} \right\|_{{L_2}(0,1)}^2 + \frac{a}{2}\left\| {\varphi \left( x \right)} \right\|_{{L_2}(0,1)}^2\,.
\end{array}
$$
Here 
$$
{T_1} = \mathop {\max }\limits_{0 < t \le T} \left( {{\varepsilon _1}2\int\limits_0^t {\psi '\left( s \right){{\left| {{{\left( {\psi \left( s \right) - \psi \left( 0 \right)} \right)}^{\beta  - 1}}E_{\alpha ,\beta }^\gamma \left( {\delta {{\left( {\psi \left( s \right) - \psi \left( 0 \right)} \right)}^\alpha }} \right)} \right|}^2}} ds} \right),
$$
$$
{T_2} = \mathop {\max }\limits_{0 < t \le T} \left( {2b{\varepsilon _2}\int\limits_0^t {\psi '\left( s \right){{\left| {\frac{{{{\left( {\psi \left( s \right) - \psi \left( 0 \right)} \right)}^{\beta  - \theta  - 1}}}}{{\Gamma \left( {\beta  - \theta } \right)}}} \right|}^2}ds} } \right),
$$
$$
{T_3} = \mathop {\max }\limits_{0 < t \le T} \left( {2\varepsilon \int\limits_0^t {\psi '\left( s \right){{\left| {{{\left( {\psi \left( s \right) - \psi \left( 0 \right)} \right)}^{\beta  - 1}}E_{\alpha ,\beta }^\gamma \left( {\delta {{\left( {\psi \left( s \right) - \psi \left( 0 \right)} \right)}^\alpha }} \right)} \right|}^2}ds} } \right),
$$
$$
{T_4} = \mathop {\max }\limits_{0 < t \le T} \left( {2\varepsilon \int\limits_0^t {\psi '\left( s \right){{\left| {\frac{{b{{\left( {\psi \left( s \right) - \psi \left( 0 \right)} \right)}^{\beta  - \theta  - 1}}}}{{\Gamma \left( {\beta  - \theta } \right)}}} \right|}^2}ds} } \right),
$$
and $T = {T_1} + {T_2} + {T_3} + {T_4}$,${b_1} = \frac{1}{{4{\varepsilon _1}}} + \frac{b}{{4{\varepsilon _2}}} + \frac{1}{{4\varepsilon }}$.

\subsection{Formal solution}

Taking advantage of the separability of variables permitted by the equation and the domain, we seek the solution in the form of a spectral expansion in the spatial variable. The associated spatial operator generates a self‑adjoint eigenvalue problem, and therefore the sought solution admits an expansion in the eigenfunctions of this operator, i.e.
\begin{eqnarray}\label{FS}
u\left( {t,x} \right) = \sum\limits_{n = 1}^{ + \infty } {{T_n}\left( t \right)\sin \pi nx} ,\,\,\,f\left( {t,x} \right) = \sum\limits_{n = 1}^{ + \infty } {{f_n}\left( t \right)\sin \pi nx} .
\end{eqnarray}
Here, ${T_n}\left( t \right)$ are the unknown functions to be determined, and  ${f_n}\left( t \right)$ are Fourier coefficients of the function $f(t,x),$ given as ${f_n}\left( t \right) = 2\int\limits_0^1 {f\left( {t,x} \right)\sin \pi nxdx}.$ 

Substituting \eqref{FS} into \eqref{eq:31} and including initial conditions, we obtain the following Cauchy-type problem:
\begin{equation}
\left\{ \begin{array}{l}
a\,{}^PD_{0t;\psi }^{\alpha ,\beta  + 1,\gamma ,\delta }{T_n}\left( t \right) + {}^PD_{0t;\psi }^{\alpha ,\beta ,\gamma ,\delta }{T_n}\left( t \right) + {(n\pi)} ^2b{}^PD_{0t;\psi }^{\alpha ,\theta ,\gamma ,\delta }{T_n}(t) + {(n\pi)} ^2{T_n}(t) = {f_n}\left( t \right),\\
\mathop {\lim }\limits_{t \to  + 0} {}^PD_{0t;\psi }^{\alpha ,\beta ,\gamma ,\delta }T(t) = {A_n}\,,\,\,\,\,\mathop {\lim }\limits_{t \to  + 0} {}^PI_{0t;\psi }^{\alpha ,1 - \beta , - \gamma ,\delta }T(t) = {B_n}\,,
\end{array} \right.
    \label{CP}
\end{equation}
where ${A_n}$ and ${B_n}$ are the Fourier coefficient of the given functions $\varphi (x)$ and $\omega (x)$, which are defined as 
$
{A_n} = 2\int\limits_0^1 {\varphi \left( x \right)\sin \pi nxdx},\,\,\, {B_n} = 2\int\limits_0^1 {\omega \left( x \right)\sin \pi nxdx}.$
 
For convenience, before solving the problem \eqref{CP}, let us prove the following lemma 

\newtheorem{lemma}{Lemma} 
\begin{lemma}
If $0<\theta <\beta <1$ and ${{\left( \psi \left( t \right)-\psi \left( 0 \right) \right)}^{1-\beta }}u\left( t,x \right)\in C\left( {{{\bar{\Omega }}}_{\left( T,1 \right)}} \right)$conditions are valid, then for the function $u\left( t,x \right)$, the following equality holds:
    \begin{equation}
        \mathop {\lim }\limits_{t \to  + 0} {}^PI_{0t;\psi }^{\alpha ,1 - \theta , - \gamma ,\delta }u(t,x) = 0\,.
    \label{equation 35}
    \end{equation}
\end{lemma}

\textbf{Proof:} Let us prove equality \eqref{equation 35}. For this aim, we rewrite the left side of \eqref{equation 35} as
$$
\begin{array}{l}
\mathop {\lim }\limits_{t \to  + 0} {}^PI_{0t;\psi }^{\alpha ,1 - \theta , - \gamma ,\delta }u(t,x) = \mathop {\lim }\limits_{t \to  + 0} \int\limits_0^t {\psi '(s){{\left( {\psi (t) - \psi (s)} \right)}^{ - \theta }}E_{\alpha ,1 - \theta }^{ - \gamma }\left( {\delta {{\left( {\psi (t) - \psi (s)} \right)}^\alpha }} \right)u(s,x)ds}  = \\
\mathop {\lim }\limits_{t \to  + 0} \int\limits_0^t {\psi '(s){{\left( {\psi (t) - \psi (s)} \right)}^{ - \theta }}{{\left( {\psi (s) - \psi (0)} \right)}^{\beta  - 1}}E_{\alpha ,1 - \theta }^{ - \gamma }\left( {\delta {{\left( {\psi (t) - \psi (s)} \right)}^\alpha }} \right){{\left( {\psi (s) - \psi (0)} \right)}^{1 - \beta }}u(s,x)ds} .
\end{array}
$$
Based on the condition ${{\left( \psi \left( t \right)-\psi \left( 0 \right) \right)}^{1-\beta }}u\left( t,x \right)\in C\left( {{{\bar{\Omega }}}_{\left( T,1 \right)}} \right)$, we can conclude that the function \\${{\left( \psi \left( t \right)-\psi \left( 0 \right) \right)}^{1-\beta }}u\left( t,x \right)$ is bounded. Since that, we get 
$$
\begin{array}{l}
\mathop {\lim }\limits_{t \to  + 0} \left| {{}^PI_{0t;\psi }^{\alpha ,1 - \theta , - \gamma ,\delta }u(t,x)} \right| = \\
\mathop {\lim }\limits_{t \to  + 0} \left| {\int\limits_0^t {\psi '(s){{\left( {\psi (t) - \psi (s)} \right)}^{ - \theta }}{{\left( {\psi (s) - \psi (0)} \right)}^{\beta  - 1}}E_{\alpha ,1 - \theta }^{ - \gamma }\left( {\delta {{\left( {\psi (t) - \psi (s)} \right)}^\alpha }} \right){{\left( {\psi (s) - \psi (0)} \right)}^{1 - \beta }}u(s,x)ds} } \right| \le \\
{\left\| {{{\left( {\psi (s) - \psi (0)} \right)}^{1 - \beta }}u(s,x)} \right\|_{C\left( {{{\bar \Omega }_{\left( {T,1} \right)}}} \right)}}\mathop {\lim }\limits_{t \to  + 0} \left| {\int\limits_0^t {\psi '(s){{\left( {\psi (t) - \psi (s)} \right)}^{ - \theta }}{{\left( {\psi (s) - \psi (0)} \right)}^{\beta  - 1}}E_{\alpha ,1 - \theta }^{ - \gamma }\left( {\delta {{\left( {\psi (t) - \psi (s)} \right)}^\alpha }} \right)ds} } \right|=\\
{\left\| {{{\left( {\psi (s) - \psi (0)} \right)}^{1 - \beta }}u(s,x)} \right\|_{C\left( {{{\bar \Omega }_{\left( {T,1} \right)}}} \right)}}\mathop {\lim }\limits_{t \to  + 0} \left| {\sum\limits_{k = 0}^{ + \infty } {\frac{{{{\left( { - \gamma } \right)}_k}{\delta ^k}}}{{\Gamma \left( {\alpha k + 1 - \theta } \right)\Gamma \left( {k + 1} \right)}}} \int\limits_0^t {\psi '(s){{\left( {\psi (t) - \psi (s)} \right)}^{\alpha k - \theta }}{{\left( {\psi (s) - \psi (0)} \right)}^{\beta  - 1}}ds} } \right|.
\end{array}
$$
Next, we change the variable of the integral by this expression ${z = \left( {\psi (t) - \psi (s)} \right)/\left( {\psi (t) - \psi (0)} \right)}$ and, using well-known properties of the Beta function, we obtain
$$
 {\left\| {{{\left( {\psi (s) - \psi (0)} \right)}^{1 - \beta }}u(s,x)} \right\|_{C\left( {{{\bar \Omega }_{\left( {T,1} \right)}}} \right)}}\mathop {\lim }\limits_{t \to  + 0} \left| {\Gamma \left( \beta  \right){{\left( {\psi (t) - \psi (0)} \right)}^{\beta  - \theta }}\sum\limits_{k = 0}^{ + \infty } {\frac{{{{\left( { - \gamma } \right)}_k}{\delta ^k}{{\left( {\psi (t) - \psi (0)} \right)}^{\alpha k}}}}{{\Gamma \left( {\alpha k + \beta  - \theta  + 1} \right)\Gamma \left( {k + 1} \right)}}} } \right| = 0.
$$

Next, we solve the problem \eqref{CP}. For this aim for convenience, introduce the following notation for the eigenvalue: 
$\lambda_{n}=(\pi n)^{2}$. With this, we reduce the equation of the Cauchy-type problem \eqref{CP} to the following form:

\begin{equation}
a\,{}^PD_{0t;\psi }^{\alpha ,\beta  + 1,\gamma ,\delta }{T_n}\left( t \right) + {}^PD_{0t;\psi }^{\alpha ,\beta ,\gamma ,\delta }{T_n}\left( t \right) + b\lambda_{n} {}^PD_{0t;\psi }^{\alpha ,\theta ,\gamma ,\delta }{T_n}(t) + \lambda_{n} {T_n}(t) = {f_n}\left( t \right).
    \label{eq:48}
\end{equation}
We apply the integral operator ${}^{P}I_{0t;\psi }^{\alpha,\beta +1,\,\gamma,\delta }$ to both sides of the equation \eqref{eq:48}, using
the formulas in \cite{O21}
$$
 {}^PI_{0t;\psi }^{\alpha ,\beta ,\,0,\delta }f(t) = {}^RI_{0t;\psi }^\beta f(t),
$$
$$
  \begin{array}{l}
{}^PI_{0t;\psi }^{\alpha ,\beta ,\gamma ,\delta }\left( {{{\left( {\psi \left( t \right) - \psi \left( 0 \right)} \right)}^{\nu  - 1}}E_{\alpha ,\nu }^\sigma \left( {\delta {{\left( {\psi \left( t \right) - \psi \left( 0 \right)} \right)}^\alpha }} \right)} \right)\\
 = {\left( {\psi \left( t \right) - \psi \left( 0 \right)} \right)^{\nu  + \beta  - 1}}E_{\alpha ,\nu  + \beta }^{\sigma  + \gamma }\left( {\delta {{\left( {\psi \left( t \right) - \psi \left( 0 \right)} \right)}^\alpha }} \right)\,,
\end{array}
$$
with $\alpha,\beta,\gamma,\delta,\nu,\sigma \in\mathbb{R}$ and $\alpha  > 0,\beta > 0,\nu  > 0,$
and equalities \eqref{eq:30}, \eqref{equation 35}, after some calculations and simplifications, we obtain the following integral equation: 

\begin{equation}
{T_n}\left( t \right) = \frac{1}{a}g(t) - \frac{\lambda_{n}}{a}{}^PI_{0t;\psi }^{\alpha ,\beta  + 1,\,\gamma ,\delta }{T_n}(t) - \frac{{b\lambda_{n}}}{a}{}^RI_{0t;\psi }^{\beta  + 1 - \theta }{T_n}\left( t \right) - \frac{1}{a}{}^RI_{0t;\psi }^1{T_n}\left( t \right).
    \label{eq:49}
\end{equation}
Here,
$$
\begin{array}{l}
g(t) = {}^PI_{0t;\psi }^{\alpha ,\beta  + 1,\,\gamma ,\delta }f_{n}\left( t \right) + \left( {a{C_1} + {C_2}} \right){\left( {\psi \left( t \right) - \psi \left( 0 \right)} \right)^\beta }E_{\alpha ,\beta  + 1}^\gamma \left( {\delta {{\left( {\psi \left( t \right) - \psi \left( 0 \right)} \right)}^\alpha }} \right) \\
+\,a{C_2}{\left( {\psi \left( t \right) - \psi \left( 0 \right)} \right)^{\beta  - 1}}E_{\alpha ,\beta }^\gamma \left( {\delta {{\left( {\psi \left( t \right) - \psi \left( 0 \right)} \right)}^\alpha }} \right)
\end{array}
$$
and ${C_1} = \mathop {\lim }\limits_{t \to  + 0} {}^PD_{0t;\psi }^{\alpha ,\beta ,\gamma ,\delta }T(t),$ ${C_2} = \mathop {\lim }\limits_{t \to  + 0} {}^PI_{0t;\psi }^{\alpha ,1 - \beta , - \gamma ,\delta }T(t)$.\\
To find a solution to the integral equation \eqref{eq:49}, we use the successive approximation method. For this aim, we consider the formula in \cite{O22}
$$
^PI_{0t;\psi }^{\alpha ,\beta ,{\kern 1pt} \gamma ,\delta }\,{\,^P}I_{0t;\psi }^{\alpha ,\nu ,{\kern 1pt} \sigma ,\delta }f(t) = {\,^P}I_{0t;\psi }^{\alpha ,\nu ,{\kern 1pt} \sigma ,\delta }\,{\,^P}I_{0t;\psi }^{\alpha ,\beta ,{\kern 1pt} \gamma ,\delta }f(t){ = ^P}I_{0t;\psi }^{\alpha ,\beta  + \nu ,{\kern 1pt} \gamma  + \sigma ,\delta }f(t),
$$
where $\alpha,\beta,\gamma,\delta,\nu,\sigma\in\mathbb{R}$ with $\alpha  > 0,\beta > 0,\nu  > 0.$\\
Let us proceed as follows:
\begin{align}
{{}_0{T_n}(t)} &= \frac{1}{a}g(t),\notag \\
{{}_m{T_n}(t)}\left( t \right) &= {{}_0{T_n}(t)} - \frac{\lambda_{n}}{a}\,\,{}^PI_{0t;\psi }^{\alpha ,\beta  + 1,\,\gamma ,\delta }{{}_{m-1}{T_n}(t)} - \frac{{b\lambda_{n}}}{a}\,\,{}^RI_{0t;\psi }^{\beta  + 1 - \theta }{{}_{m-1}{T_n}(t)} - \frac{1}{a}\,\,{}^RI_{0t;\psi }^1\,\,{{}_{m-1}{T_n}(t)},\notag\\
{{}_1{T_n}(t)} &= {{}_0{T_n}(t)} - \frac{\lambda_{n}}{a}\,\,{}^PI_{0t;\psi }^{\alpha ,\beta  + 1,\,\gamma ,\delta }{{}_0{T_n}(t)} - \frac{{b\lambda_{n} }}{a}\,\,{}^PI_{0t;\psi }^{\alpha ,\beta  + 1 - \theta ,0,\delta }{{}_0{T_n}(t)} - \frac{1}{a}\,\,{}^PI_{0t;\psi }^{\alpha ,1,0,\delta }{{}_0{T_n}(t)} \notag\\
 & = {\sum\limits_{k = 0}^1 {\left( { - \frac{\lambda_{n}}{a}\,\,{}^PI_{0t;\psi }^{\alpha ,\beta  + 1,\,\gamma ,\delta } - \frac{{b\lambda_{n}}}{a}\,\,{}^PI_{0t;\psi }^{\alpha ,\beta  + 1 - \theta ,0,\delta } - \frac{1}{a}\,\,{}^PI_{0t;\psi }^{\alpha ,1,0,\delta }} \right)} ^k}{{}_0{T_n}(t)}\notag\,,
\end{align}

\begin{align}
    {}&{{}_2{T_n}(t)}\left( t \right) = {{}_0{T_n}(t)} - \frac{\lambda_{n}}{a}{}^PI_{0t;\psi }^{\alpha ,\beta  + 1,\,\gamma ,\delta }{{}_1{T_n}(t)} - \frac{{b\lambda_{n}}}{a}{}^PI_{0t;\psi }^{\alpha ,\beta  + 1 - \theta ,0,\delta }{{}_1{T_n}(t)} - \frac{1}{a}{}^PI_{0t;\psi }^{\alpha ,1,0,\delta }{{}_1{T_n}(t)}\notag\\
     &= {{}_0{T_n}(t)} - \frac{\lambda_{n}}{a}{}^PI_{0t;\psi }^{\alpha ,\beta  + 1,\,\gamma ,\delta }\left[ {{{}_0{T_n}(t)} - \frac{\lambda_{n}}{a}{}^PI_{0t;\psi }^{\alpha ,\beta  + 1,\,\gamma ,\delta }{{}_0{T_n}(t)} - \frac{{b\lambda_{n}}}{a}{}^PI_{0t;\psi }^{\alpha ,\beta  + 1 - \theta ,0,\delta }{{}_0{T_n}(t)} - \frac{1}{a}{}^PI_{0t;\psi }^{\alpha ,1,0,\delta }{{}_0{T_n}(t)}} \right]\notag\\
     &- \frac{{b\lambda_{n}}}{a}{}^PI_{0t;\psi }^{\alpha ,\beta  + 1 - \theta ,0,\delta }\left[ {{{}_0{T_n}(t)} - \frac{\lambda_{n}}{a}{}^PI_{0t;\psi }^{\alpha ,\beta  + 1,\,\gamma ,\delta }{{}_0{T_n}(t)} - \frac{{b\lambda_{n}}}{a}{}^PI_{0t;\psi }^{\alpha ,\beta  + 1 - \theta ,0,\delta }{{}_0{T_n}(t)} - \frac{1}{a}{}^PI_{0t;\psi }^{\alpha ,1,0,\delta }{{}_0{T_n}(t)}} \right]\notag \\
      &- \frac{1}{a}{}^PI_{0t;\psi }^{\alpha ,1,0,\delta }\left[ {{{}_0{T_n}(t)} - \frac{\lambda_{n}}{a}{}^PI_{0t;\psi }^{\alpha ,\beta  + 1,\,\gamma ,\delta }{{}_0{T_n}(t)} - \frac{{b\lambda_{n}}}{a}{}^PI_{0t;\psi }^{\alpha ,\beta  + 1 - \theta ,0,\delta }{{}_0{T_n}(t)} - \frac{1}{a}{}^PI_{0t;\psi }^{\alpha ,1,0,\delta }{{}_0{T_n}(t)}} \right]\notag \\
      &= {{}_0{T_n}(t)} - \frac{\lambda_{n}}{a}{}^PI_{0t;\psi }^{\alpha ,\beta  + 1,\,\gamma ,\delta }{{}_0{T_n}(t)} + \frac{{{\lambda_{n}^{2}}}}{{{a^2}}}{}^PI_{0t;\psi }^{\alpha ,2\beta  + 2,\,2\gamma ,\delta }{{}_0{T_n}(t)} + \frac{{b{\lambda_{n}^2}}}{{{a^2}}}{}^PI_{0t;\psi }^{\alpha ,2\beta  + 2 - \theta ,\gamma ,\delta }{{}_0{T_n}(t)} \notag\\
      &+ \frac{\lambda_{n}}{{{a^2}}}{}^PI_{0t;\psi }^{\alpha ,\beta  + 2,\gamma ,\delta }{{}_0{T_n}(t)} - \frac{{b\lambda_{n}}}{a}{}^PI_{0t;\psi }^{\alpha ,\beta  + 1 - \theta ,0,\delta }{{}_0{T_n}(t)} + \frac{{b{\lambda_{n}^2}}}{{{a^2}}}{}^PI_{0t;\psi }^{\alpha ,2\beta  + 2 - \theta ,\gamma ,\delta }{{}_0{T_n}(t)} + \frac{{{b^2}{\lambda_{n}^2}}}{{{a^2}}}{}^PI_{0t;\psi }^{\alpha ,2\beta  + 2 - 2\theta ,0,\delta }{{}_0{T_n}(t)}\notag\\
      &+ \frac{{b\lambda_{n}}}{{{a^2}}}{}^PI_{0t;\psi }^{\alpha ,\beta  + 2 - \theta ,0,\delta }{{}_0{T_n}(t)} - \frac{1}{a}{}^PI_{0t;\psi }^{\alpha ,1,0,\delta }{{}_0{T_n}(t)} + \frac{\lambda_{n}}{{{a^2}}}{}^PI_{0t;\psi }^{\alpha ,\beta  + 2,\gamma ,\delta }{{}_0{T_n}(t)} + \frac{{b\lambda_{n}}}{{{a^2}}}{}^PI_{0t;\psi }^{\alpha ,\beta  + 2 - \theta ,0,\delta }{{}_0{T_n}(t)} \notag \\ 
      &+ \frac{1}{{{a^2}}}{}^PI_{0t;\psi }^{\alpha ,2,0,\delta }{{}_0{T_n}(t)}
       = {{}_0{T_n}(t)} + \left( { - \frac{\lambda_{n}}{a}{}^PI_{0t;\psi }^{\alpha ,\beta  + 1,\,\gamma ,\delta }{{}_0{T_n}(t)} - \frac{{b\lambda_{n}}}{a}{}^PI_{0t;\psi }^{\alpha ,\beta  + 1 - \theta ,0,\delta }{{}_0{T_n}(t)} - \frac{1}{a}{}^PI_{0t;\psi }^{\alpha ,1,0,\delta }{{}_0{T_n}(t)}} \right)\notag \\
       &+ \left( {\frac{{{\lambda_{n}^2}}}{{{a^2}}}{}^PI_{0t;\psi }^{\alpha ,2\beta  + 2,\,2\gamma ,\delta }{{}_0{T_n}(t)} + \frac{{{b^2}{\lambda_{n}^2}}}{{{a^2}}}{}^PI_{0t;\psi }^{\alpha ,2\beta  + 2 - 2\theta ,0,\delta }{{}_0{T_n}(t)} + \frac{1}{{{a^2}}}{}^PI_{0t;\psi }^{\alpha ,2,0,\delta }{{}_0{T_n}(t)}} \right.\notag \\
        &+ \left. {2\frac{{b{\lambda_{n}^2}}}{{{a^2}}}{}^PI_{0t;\psi }^{\alpha ,2\beta  + 2 - \theta ,\gamma ,\delta }{{}_0{T_n}(t)} + 2\frac{\lambda_{n}}{{{a^2}}}{}^PI_{0t;\psi }^{\alpha ,\beta  + 2,\gamma ,\delta }{{}_0{T_n}(t)} + 2\frac{{b\lambda_{n}}}{{{a^2}}}{}^PI_{0t;\psi }^{\alpha ,\beta  + 2 - \theta ,0,\delta }{{}_0{T_n}(t)}} \right)\notag \\
         &= {\sum\limits_{k = 0}^2 {\left( { - \frac{\lambda_{n}}{a}{}^PI_{0t;\psi }^{\alpha ,\beta  + 1,\,\gamma ,\delta } - \frac{{b\lambda_{n}}}{a}{}^PI_{0t;\psi }^{\alpha ,\beta  + 1 - \theta ,0,\delta } - \frac{1}{a}{}^PI_{0t;\psi }^{\alpha ,1,0,\delta }} \right)} ^k}{{}_0{T_n}(t)}\,.\notag
\end{align}
Similarly, if we continue this process $m$ times, then we get
\begin{equation}
{{}_m{T_n}(t)} = {\sum\limits_{k = 0}^m {\left( { - \frac{\lambda_{n}}{a}{}^PI_{0t;\psi }^{\alpha ,\beta  + 1,\,\gamma ,\delta } - \frac{{b\lambda_{n}}}{a}{}^PI_{0t;\psi }^{\alpha ,\beta  + 1 - \theta ,0,\delta } - \frac{1}{a}{}^PI_{0t;\psi }^{\alpha ,1,0,\delta }} \right)} ^k}{{}_0{T_n}(t)}\,.
    \label{eq:50}
\end{equation}
If we apply the formula in \cite{Aigner97} 
$$
{\left( {a + b + c} \right)^n} = \sum\limits_{k = 0}^n {\sum\limits_{j = 0}^k {\frac{{\Gamma \left( {n + 1} \right)}}{{\Gamma \left( {n - k + 1} \right)\Gamma \left( {k - j + 1} \right)\Gamma \left( {j + 1} \right)}}{a^{n - k}}{b^{k - j}}{c^j}} } 
$$
to \eqref{eq:50}, then we obtain
\begin{align}
    {{}_m{T_n}(t)} &= \sum\limits_{k = 0}^m {\sum\limits_{i = 0}^k {\sum\limits_{j = 0}^i {\frac{{\Gamma \left( {k + 1} \right){{\left( { - \frac{\lambda_{n}}{a}} \right)}^{k - i}}{{\left( { - \frac{{b\lambda_{n}}}{a}} \right)}^{i - j}}{{\left( { - \frac{1}{a}} \right)}^j}}}{{\Gamma \left( {k - i + 1} \right)\Gamma \left( {i - j + 1} \right)\Gamma \left( {j + 1} \right)}}} } } \notag \\
    &\times \,\,{}^PI_{0t;\psi }^{\alpha ,\left( {\beta  + 1} \right)\left( {k - i} \right),\,\gamma \left( {k - i} \right),\delta }{}^PI_{0t;\psi }^{\alpha ,\left( {\beta  + 1 - \theta } \right)\left( {i - j} \right),0,\delta }{}^PI_{0t;\psi }^{\alpha ,j,0,\delta }{{{}_0{T_n}(t)}}\left( t \right)\,.\notag
\end{align}
If we calculate limit of the ${{}_m{T_n}(t)}$ when $m\to +\infty $ we obtain the  ${{T}_{n}}\left( t \right)$ as follows:
\begin{align}
    T_{n}\left( t \right) &= \sum\limits_{k = 0}^{ + \infty } {\sum\limits_{i = 0}^k {\sum\limits_{j = 0}^i {\frac{{\Gamma \left( {k + 1} \right){{\left( { - \frac{\lambda_{n}}{a}} \right)}^{k - i}}{{\left( { - \frac{{b\lambda_{n}}}{a}} \right)}^{i - j}}{{\left( { - \frac{1}{a}} \right)}^j}}}{{\Gamma \left( {k - i + 1} \right)\Gamma \left( {i - j + 1} \right)\Gamma \left( {j + 1} \right)}}} } } \,\notag \\
    &\times \,\,{}^PI_{0t;\psi }^{\alpha ,\left( {\beta  + 1} \right)\left( {k - i} \right),\,\gamma \left( {k - i} \right),\delta }{}^PI_{0t;\psi }^{\alpha ,\left( {\beta  + 1 - \theta } \right)\left( {i - j} \right),0,\delta }{}^PI_{0t;\psi }^{\alpha ,j,0,\delta }{{{}_0{T_n}(t)}}\left( t \right). \label{eq:53}
\end{align}
We know from \cite{Dummit2004}that the following formula holds:
\begin{equation}
 \sum\limits_{r = 0}^{ + \infty } {\sum\limits_{h = 0}^r {{a_h}{b_{r - h}}} } =\sum\limits_{n = 0}^{ + \infty }  \sum\limits_{m = 0}^{ + \infty } {{a_n}}{{b_m}} .
    \label{eq:54}
\end{equation}
Here, used the following remarks $n=h,r=m+n$. Using two-time, the formula \eqref{eq:54}, we obtain the following result:
\begin{equation}
\sum\limits_{n = 0}^{ + \infty } {\sum\limits_{m = 0}^n {\sum\limits_{k = 0}^m {{a_k}{b_{m - k}}{c_{n - m}}} } }=\sum\limits_{p = 0}^{ + \infty }  \sum\limits_{q = 0}^{ + \infty }  \sum\limits_{z = 0}^{ + \infty }{{a_p}}{{b_q}} {{c_z}}  \,.
    \label{eq:55}
\end{equation}
Here, $n - m = p,\,\,k = q,\,n - p - q = z\,.$ If we apply the formula \eqref{eq:55} to \eqref{eq:53}, we have
$$
T_{n}\left( t \right) = \sum\limits_{k = 0}^{ + \infty } {\sum\limits_{i = 0}^{ + \infty } {\sum\limits_{j = 0}^{ + \infty } {\frac{{\Gamma \left( {k + i + j + 1} \right){{\left( { - \frac{\lambda_{n}}{a}} \right)}^k}{{\left( { - \frac{{b\lambda_{n}}}{a}} \right)}^i}{{\left( { - \frac{1}{a}} \right)}^j}}}{{\Gamma \left( {k + 1} \right)\Gamma \left( {i + 1} \right)\Gamma \left( {j + 1} \right)}}} } } \,{}^PI_{0t;\psi }^{\alpha ,\left( {\beta  + 1} \right)k,\,\gamma k,\delta }{}^PI_{0t;\psi }^{\alpha ,\left( {\beta  + 1 - \theta } \right)i,0,\delta }{}^PI_{0t;\psi }^{\alpha ,j,0,\delta }{{{}_0{T_n}(t)}}\left( t \right)\,.
$$
Since that,
\begin{equation}
T_{n}\left( t \right) = \sum\limits_{k = 0}^{ + \infty } {\sum\limits_{i = 0}^{ + \infty } {\sum\limits_{j = 0}^{ + \infty } {\frac{{\Gamma \left( {k + i + j + 1} \right){{\left( { - \frac{\lambda_{n}}{a}} \right)}^k}{{\left( { - \frac{{b\lambda_{n}}}{a}} \right)}^i}{{\left( { - \frac{1}{a}} \right)}^j}\left( {\frac{1}{a}} \right)}}{{\Gamma \left( {k + 1} \right)\Gamma \left( {i + 1} \right)\Gamma \left( {j + 1} \right)}}{}^PI_{0t;\psi }^{\alpha ,\left( {\beta  + 1} \right)k + \left( {\beta  + 1 - \theta } \right)i + j,\,\gamma k,\delta }} } } g\left( t \right)\,.
    \label{eq:56}
\end{equation}
We replace the function $g(t)$in expression \eqref{eq:56} with an appropriate term, and then, after some simplifications, we obtain:
\begin{align}
    {T_n}\left( t \right) &= {C_{2}} {\left( {\psi \left( t \right) - \psi \left( 0 \right)} \right)^{\beta  - 1}}\sum\limits_{i = 0}^{ + \infty } {\sum\limits_{j = 0}^{ + \infty } {\sum\limits_{k = 0}^{ + \infty } {\sum\limits_{m = 0}^{ + \infty } {\frac{{\Gamma \left( {i + j + k + 1} \right){{\left( {\gamma k + \gamma } \right)}_m}}}{{\Gamma \left( {\left( {\beta  + 1 - \theta } \right)i + j + \left( {\beta  + 1} \right)k + \alpha m + \beta } \right)}}\,\,} } } } \notag \\
    &\times \frac{{{{\left( { - \frac{{b\lambda_{n}}}{a}{{\left( {\psi \left( t \right) - \psi \left( 0 \right)} \right)}^{\beta  + 1 - \theta }}} \right)}^i}}}{{\Gamma \left( {i + 1} \right)}}\frac{{{{\left( { - \frac{1}{a}\left( {\psi \left( t \right) - \psi \left( 0 \right)} \right)} \right)}^j}}}{{\Gamma \left( {j + 1} \right)}}\frac{{{{\left( { - \frac{\lambda_{n}}{a}{{\left( {\psi \left( t \right) - \psi \left( 0 \right)} \right)}^{\beta  + 1}}} \right)}^k}}}{{\Gamma \left( {k + 1} \right)}}\frac{{{{\left( {\delta {{\left( {\psi \left( t \right) - \psi \left( 0 \right)} \right)}^\alpha }} \right)}^m}}}{{\Gamma \left( {m + 1} \right)}}\notag\\
    &+ \frac{{\left( {a{C_{1}} + {C_{2}}} \right)}}{a}{\left( {\psi \left( t \right) - \psi \left( 0 \right)} \right)^\beta }\sum\limits_{i = 0}^{ + \infty } {\sum\limits_{j = 0}^{ + \infty } {\sum\limits_{k = 0}^{ + \infty } {\sum\limits_{m = 0}^{ + \infty } {\frac{{\Gamma \left( {i + j + k + 1} \right){{\left( {\gamma k + \gamma } \right)}_m}}}{{\Gamma \left( {\left( {\beta  + 1 - \theta } \right)i + j + \left( {\beta  + 1} \right)k + \alpha m + \beta  + 1} \right)}}\,\,} } } } \notag \\
    &\times \frac{{{{\left( { - \frac{{b\lambda_{n}}}{a}{{\left( {\psi \left( t \right) - \psi \left( 0 \right)} \right)}^{\beta  + 1 - \theta }}} \right)}^i}}}{{\Gamma \left( {i + 1} \right)}}\frac{{{{\left( { - \frac{1}{a}\left( {\psi \left( t \right) - \psi \left( 0 \right)} \right)} \right)}^j}}}{{\Gamma \left( {j + 1} \right)}}\frac{{{{\left( { - \frac{\lambda_{n}}{a}{{\left( {\psi \left( t \right) - \psi \left( 0 \right)} \right)}^{\beta  + 1}}} \right)}^k}}}{{\Gamma \left( {k + 1} \right)}}\frac{{{{\left( {\delta {{\left( {\psi \left( t \right) - \psi \left( 0 \right)} \right)}^\alpha }} \right)}^m}}}{{\Gamma \left( {m + 1} \right)}}\notag\\
    &+ \frac{1}{a}\int\limits_0^t {{{\left( {\psi \left( t \right) - \psi \left( s \right)} \right)}^\beta }\sum\limits_{i = 0}^{ + \infty } {\sum\limits_{j = 0}^{ + \infty } {\sum\limits_{k = 0}^{ + \infty } {\sum\limits_{m = 0}^{ + \infty } {\frac{{\Gamma \left( {i + j + k + 1} \right){{\left( {\gamma k + \gamma } \right)}_m}\,\,{{\left( { - \frac{{b\lambda_{n}}}{a}{{\left( {\psi \left( t \right) - \psi \left( s \right)} \right)}^{\beta  + 1 - \theta }}} \right)}^i}}}{{\Gamma \left( {\left( {\beta  + 1 - \theta } \right)i + j + \left( {\beta  + 1} \right)k + \alpha m + \beta  + 1} \right)\Gamma \left( {i + 1} \right)}}\,\,} } } } } \notag\\
    &\times \frac{{{{\left( { - \frac{1}{a}\left( {\psi \left( t \right) - \psi \left( s \right)} \right)} \right)}^j}}}{{\Gamma \left( {j + 1} \right)}}\frac{{{{\left( { - \frac{\lambda_{n}}{a}{{\left( {\psi \left( t \right) - \psi \left( s \right)} \right)}^{\beta  + 1}}} \right)}^k}}}{{\Gamma \left( {k + 1} \right)}}\frac{{{{\left( {\delta {{\left( {\psi \left( t \right) - \psi \left( s \right)} \right)}^\alpha }} \right)}^m}}}{{\Gamma \left( {m + 1} \right)}}\psi '\left( s \right)f_{n}\left( s \right)ds\,.\notag
\end{align}
Considering the initial conditions, we find the solution ${{T}_{n}}\left( t \right)$ of the problem \eqref{CP}  as follows: 

\begin{align*}
    &{T_n}\left( t \right) = {B_n}{\left( {\psi \left( t \right) - \psi \left( 0 \right)} \right)^{\beta  - 1}} \times \notag\\
    &{}_4{E_1}\left( {\begin{array}{l}
1,1,1,1;\gamma ,1,\gamma;\\
\beta  + 1 - \theta ,1,\beta  + 1,\alpha ,\beta ;\gamma ,\gamma ;1,1;1,1;1,1;1,1;
\end{array}}
 \mathrel{\left | {\vphantom {\begin{array}{l}
1,1,1,1;\gamma ,1,\gamma;\\
\beta  + 1 - \theta ,1,\beta  + 1,\alpha ,\beta ;\gamma ,\gamma ;1,1;1,1;1,1;1,1;
\end{array} \begin{array}{l}
\left( { - b\lambda_n \frac{{{{\left( {\psi \left( t \right) - \psi \left( 0 \right)} \right)}^{\beta  + 1 - \theta }}}}{a}} \right);\left( { - \frac{{\left( {\psi \left( t \right) - \psi \left( 0 \right)} \right)}}{a}} \right)\\
\left( { - \lambda_n \frac{{{{\left( {\psi \left( t \right) - \psi \left( 0 \right)} \right)}^{\beta  + 1}}}}{a}} \right);\left( {\delta {{\left( {\psi \left( t \right) - \psi \left( 0 \right)} \right)}^\alpha }} \right)
\end{array}}}
 \right. \kern-\nulldelimiterspace}
 {\begin{array}{l}
\left( { - b\lambda_n \frac{{{{\left( {\psi \left( t \right) - \psi \left( 0 \right)} \right)}^{\beta  + 1 - \theta }}}}{a}} \right);\left( { - \frac{{\left( {\psi \left( t \right) - \psi \left( 0 \right)} \right)}}{a}} \right)\\
\left( { - \lambda_n \frac{{{{\left( {\psi \left( t \right) - \psi \left( 0 \right)} \right)}^{\beta  + 1}}}}{a}} \right);\left( {\delta {{\left( {\psi \left( t \right) - \psi \left( 0 \right)} \right)}^\alpha }} \right)
\end{array}} \right)+ \notag \\
& \frac{{\left( {a{A_n} + {B_n}} \right)}}{a}{\left( {\psi \left( t \right) - \psi \left( 0 \right)} \right)^\beta } \times \notag \\
&{}_4{E_1}\left( {\begin{array}{l}
1,1,1,1;\gamma ,1,\gamma;\\
\beta  + 1 - \theta ,1,\beta  + 1,\alpha ,\beta  + 1;\gamma ,\gamma ;1,1;1,1;1,1;1,1;
\end{array}}
 \mathrel{\left | {\vphantom {\begin{array}{l}
1,1,1,1;\gamma ,1,\gamma;\\
\beta  + 1 - \theta ,1,\beta  + 1,\alpha ,\beta  + 1;\gamma ,\gamma ;1,1;1,1;1,1;1,1;
\end{array} \begin{array}{l}
\left( { - b\lambda_n \frac{{{{\left( {\psi \left( t \right) - \psi \left( 0 \right)} \right)}^{\beta  + 1 - \theta }}}}{a}} \right);\left( { - \frac{{\left( {\psi \left( t \right) - \psi \left( 0 \right)} \right)}}{a}} \right)\\
\left( { - \lambda_n \frac{{{{\left( {\psi \left( t \right) - \psi \left( 0 \right)} \right)}^{\beta  + 1}}}}{a}} \right);\left( {\delta {{\left( {\psi \left( t \right) - \psi \left( 0 \right)} \right)}^\alpha }} \right)
\end{array}}}
 \right. \kern-\nulldelimiterspace}
 {\begin{array}{l}
\left( { - b\lambda_n \frac{{{{\left( {\psi \left( t \right) - \psi \left( 0 \right)} \right)}^{\beta  + 1 - \theta }}}}{a}} \right);\left( { - \frac{{\left( {\psi \left( t \right) - \psi \left( 0 \right)} \right)}}{a}} \right)\\
\left( { - \lambda_n \frac{{{{\left( {\psi \left( t \right) - \psi \left( 0 \right)} \right)}^{\beta  + 1}}}}{a}} \right);\left( {\delta {{\left( {\psi \left( t \right) - \psi \left( 0 \right)} \right)}^\alpha }} \right)
\end{array}} \right)+ \notag \\
& \frac{1}{a}\int\limits_0^t {{{\left( {\psi \left( t \right) - \psi \left( s \right)} \right)}^\beta }}  \times \notag \\
&{}_4{E_1}\left( {\begin{array}{l}
1,1,1,1;\gamma ,1,\gamma;\\
\beta  + 1 - \theta ,1,\beta  + 1,\alpha ,\beta  + 1;\gamma ,\gamma ;1,1;1,1;1,1;1,1;
\end{array}}
 \mathrel{\left | {\vphantom {\begin{array}{l}
1,1,1,1;\gamma ,1,\gamma;\\
\beta  + 1 - \theta ,1,\beta  + 1,\alpha ,\beta  + 1;\gamma ,\gamma ;1,1;1,1;1,1;1,1;
\end{array} \begin{array}{l}
\left( { - b\lambda_n \frac{{{{\left( {\psi \left( t \right) - \psi \left( s \right)} \right)}^{\beta  + 1 - \theta }}}}{a}} \right);\left( { - \frac{{\left( {\psi \left( t \right) - \psi \left( s \right)} \right)}}{a}} \right)\\
\left( { - \lambda_n \frac{{{{\left( {\psi \left( t \right) - \psi \left( s \right)} \right)}^{\beta  + 1}}}}{a}} \right);\left( {\delta {{\left( {\psi \left( t \right) - \psi \left( s \right)} \right)}^\alpha }} \right)
\end{array}}}
 \right. \kern-\nulldelimiterspace}
 {\begin{array}{l}
\left( { - b\lambda_n \frac{{{{\left( {\psi \left( t \right) - \psi \left( s \right)} \right)}^{\beta  + 1 - \theta }}}}{a}} \right);\left( { - \frac{{\left( {\psi \left( t \right) - \psi \left( s \right)} \right)}}{a}} \right)\\
\left( { - \lambda_n \frac{{{{\left( {\psi \left( t \right) - \psi \left( s \right)} \right)}^{\beta  + 1}}}}{a}} \right);\left( {\delta {{\left( {\psi \left( t \right) - \psi \left( s \right)} \right)}^\alpha }} \right)
\end{array}} \right) \times\notag \\
& \,\,\psi '\left( s \right)f_n\left( s \right)ds,
\end{align*}
where 
\begin{align}
&{}_4E_1\left(
\begin{array}{l}
\alpha_1,\beta_1,\gamma_1,\delta_1;\,\gamma_2,\theta_1,\delta_2; \\
\alpha_2,\beta_2,\gamma_3,\theta_2,\delta_3;\,\gamma_4,\delta_4;\,\alpha_3,\delta_5;\,\beta_3,\delta_6;\,\gamma_5,\delta_7;\,\theta_3,\delta_8;
\end{array}
\mathrel{\left|\vphantom{
\begin{array}{l}
\alpha_1,\beta_1,\gamma_1,\delta_1;\,\gamma_2,\theta_1,\delta_2; \\
\alpha_2,\beta_2,\gamma_3,\theta_2,\delta_3;\,\gamma_4,\delta_4;\,\alpha_3,\delta_5;\,\beta_3,\delta_6;\,\gamma_5,\delta_7;\,\theta_3,\delta_8;
\end{array}
\begin{array}{l}
x;y\\
z;t
\end{array}}\right.}
\begin{array}{l}
x;y\\
z;t
\end{array}
\right)= \notag\\
& \sum_{i=0}^{+\infty} \sum_{j=0}^{+\infty} \sum_{k=0}^{+\infty} \sum_{m=0}^{+\infty}
\frac{
\Gamma\left(\alpha_1 i + \beta_1 j + \gamma_1 k + \delta_1\right)
\Gamma\left(\gamma_2 k + \theta_1 m + \delta_2\right)
}{
\Gamma\left(\alpha_2 i + \beta_2 j + \gamma_3 k + \theta_2 m + \delta_3\right)
\Gamma\left(\gamma_4 k + \delta_4\right)
}\times \notag \\
& \,
\frac{x^i}{\Gamma\left(\alpha_3 i + \delta_5\right)}
\frac{y^j}{\Gamma\left(\beta_3 j + \delta_6\right)}
\frac{z^k}{\Gamma\left(\gamma_5 k + \delta_7\right)}
\frac{t^m}{\Gamma\left(\theta_3 m + \delta_8\right)}
\label{SF}
\end{align}
is a quadrivariate Mittag-Leffler-type function introduced and studied in \cite{KKh26}.

Consequently, including the equality $\lambda_{n}=(\pi n)^{2}$, the solution $u(t,x)$ of the main problem can be represented as follows:
\begin{align}
    &u(t,x) = \sum\limits_{n = 1}^{ + \infty } {\left[ {{B_n}{{\left( {\psi \left( t \right) - \psi \left( 0 \right)} \right)}^{\beta  - 1}} \times } \right.}\notag\\
&{}_4{E_1}\left( {\begin{array}{l}
1,1,1,1;\gamma ,1,\gamma;\\
\beta  + 1 - \theta ,1,\beta  + 1,\alpha ,\beta ;\gamma ,\gamma ;1,1;1,1;1,1;1,1;
\end{array}}
 \mathrel{\left | {\vphantom {\begin{array}{l}
1,1,1,1;\gamma ,1,\gamma;\\
\beta  + 1 - \theta ,1,\beta  + 1,\alpha ,\beta ;\gamma ,\gamma ;1,1;1,1;1,1;1,1;
\end{array} \begin{array}{l}
\left( { - b{{\left( {\pi n} \right)}^2}\frac{{{{\left( {\psi \left( t \right) - \psi \left( 0 \right)} \right)}^{\beta  + 1 - \theta }}}}{a}} \right);\left( { - \frac{{\left( {\psi \left( t \right) - \psi \left( 0 \right)} \right)}}{a}} \right)\\
\left( { - {{\left( {\pi n} \right)}^2}\frac{{{{\left( {\psi \left( t \right) - \psi \left( 0 \right)} \right)}^{\beta  + 1}}}}{a}} \right);\left( {\delta {{\left( {\psi \left( t \right) - \psi \left( 0 \right)} \right)}^\alpha }} \right)
\end{array}}}
 \right. \kern-\nulldelimiterspace}
 {\begin{array}{l}
\left( { - b{{\left( {\pi n} \right)}^2}\frac{{{{\left( {\psi \left( t \right) - \psi \left( 0 \right)} \right)}^{\beta  + 1 - \theta }}}}{a}} \right);\left( { - \frac{{\left( {\psi \left( t \right) - \psi \left( 0 \right)} \right)}}{a}} \right)\\
\left( { - {{\left( {\pi n} \right)}^2}\frac{{{{\left( {\psi \left( t \right) - \psi \left( 0 \right)} \right)}^{\beta  + 1}}}}{a}} \right);\left( {\delta {{\left( {\psi \left( t \right) - \psi \left( 0 \right)} \right)}^\alpha }} \right)
\end{array}} \right) \notag \\
&+ \frac{{\left( {a{A_n} + {B_n}} \right)}}{a}{\left( {\psi \left( t \right) - \psi \left( 0 \right)} \right)^\beta } \times \notag \\
&{}_4{E_1}\left( {\begin{array}{l}
1,1,1,1;\gamma ,1,\gamma;\\
\beta  + 1 - \theta ,1,\beta  + 1,\alpha ,\beta  + 1;\gamma ,\gamma ;1,1;1,1;1,1;1,1;
\end{array}}
 \mathrel{\left | {\vphantom {\begin{array}{l}
1,1,1,1;\gamma ,1,\gamma;\\
\beta  + 1 - \theta ,1,\beta  + 1,\alpha ,\beta  + 1;\gamma ,\gamma ;1,1;1,1;1,1;1,1;
\end{array} \begin{array}{l}
\left( { - b{{\left( {\pi n} \right)}^2}\frac{{{{\left( {\psi \left( t \right) - \psi \left( 0 \right)} \right)}^{\beta  + 1 - \theta }}}}{a}} \right);\left( { - \frac{{\left( {\psi \left( t \right) - \psi \left( 0 \right)} \right)}}{a}} \right)\\
\left( { - {{\left( {\pi n} \right)}^2}\frac{{{{\left( {\psi \left( t \right) - \psi \left( 0 \right)} \right)}^{\beta  + 1}}}}{a}} \right);\left( {\delta {{\left( {\psi \left( t \right) - \psi \left( 0 \right)} \right)}^\alpha }} \right)
\end{array}}}
 \right. \kern-\nulldelimiterspace}
 {\begin{array}{l}
\left( { - b{{\left( {\pi n} \right)}^2}\frac{{{{\left( {\psi \left( t \right) - \psi \left( 0 \right)} \right)}^{\beta  + 1 - \theta }}}}{a}} \right);\left( { - \frac{{\left( {\psi \left( t \right) - \psi \left( 0 \right)} \right)}}{a}} \right)\\
\left( { - {{\left( {\pi n} \right)}^2}\frac{{{{\left( {\psi \left( t \right) - \psi \left( 0 \right)} \right)}^{\beta  + 1}}}}{a}} \right);\left( {\delta {{\left( {\psi \left( t \right) - \psi \left( 0 \right)} \right)}^\alpha }} \right)
\end{array}} \right) \notag \\
&+ \frac{1}{a}\int\limits_0^t {{{\left( {\psi \left( t \right) - \psi \left( s \right)} \right)}^\beta } \times } \notag \\
&{}_4{E_1}\left( {\begin{array}{l}
1,1,1,1;\gamma ,1,\gamma;\\
\beta  + 1 - \theta ,1,\beta  + 1,\alpha ,\beta  + 1;\gamma ,\gamma ;1,1;1,1;1,1;1,1;
\end{array}}
 \mathrel{\left | {\vphantom {\begin{array}{l}
1,1,1,1;\gamma ,1,\gamma;\\
\beta  + 1 - \theta ,1,\beta  + 1,\alpha ,\beta  + 1;\gamma ,\gamma ;1,1;1,1;1,1;1,1;
\end{array} \begin{array}{l}
\left( { - b{{\left( {\pi n} \right)}^2}\frac{{{{\left( {\psi \left( t \right) - \psi \left( s \right)} \right)}^{\beta  + 1 - \theta }}}}{a}} \right);\left( { - \frac{{\left( {\psi \left( t \right) - \psi \left( s \right)} \right)}}{a}} \right)\\
\left( { - {{\left( {\pi n} \right)}^2}\frac{{{{\left( {\psi \left( t \right) - \psi \left( s \right)} \right)}^{\beta  + 1}}}}{a}} \right);\left( {\delta {{\left( {\psi \left( t \right) - \psi \left( s \right)} \right)}^\alpha }} \right)
\end{array}}}
 \right. \kern-\nulldelimiterspace}
 {\begin{array}{l}
\left( { - b{{\left( {\pi n} \right)}^2}\frac{{{{\left( {\psi \left( t \right) - \psi \left( s \right)} \right)}^{\beta  + 1 - \theta }}}}{a}} \right);\left( { - \frac{{\left( {\psi \left( t \right) - \psi \left( s \right)} \right)}}{a}} \right)\\
\left( { - {{\left( {\pi n} \right)}^2}\frac{{{{\left( {\psi \left( t \right) - \psi \left( s \right)} \right)}^{\beta  + 1}}}}{a}} \right);\left( {\delta {{\left( {\psi \left( t \right) - \psi \left( s \right)} \right)}^\alpha }} \right)
\end{array}} \right) \notag \\
&\left. {\psi '\left( s \right){f_n}\left( s \right)ds\,} \right]\sin \pi nx\,.\label{eq:57}
\end{align}

\subsection{Convergence of infinite series and illustrating example}
The next statement is related with the uniform convergence of infinite series given by \eqref{FS} and the other infinite series corresponding to the partial derivatives of $u(t,x)$ involved in the main equation.

\begin{theorem}
    If \,\,$0 < \beta  - \theta  + \frac{{2{\varepsilon _1}}}{\pi } < 2,\,0 < \beta  + 1 - \gamma  + \frac{{2{\varepsilon _2}}}{\pi } < 2,\,0 < 1 - \alpha  + \frac{{2{\varepsilon _3}}}{\pi } < 2,\,0 < 1 - \beta  + \frac{{2{\varepsilon _4}}}{\pi } < 2,\,\\0 < \gamma  < 1\,,$ ${\varphi ^{\left( i \right)}}\left( 0 \right) = {\varphi ^{\left( i \right)}}\left( 1 \right)=0,\,\omega^{(i)}(0)=\omega^{(i)}(1)=0$,$(i=0,2,4,6,8)$, $\varphi (x),\omega(x) \in C^{8}[0,1],$ $\varphi^{(9)}(x), \omega^{(9)}(x) \in L_{2}(0,1),$
    $\frac{{{\partial ^k}f\left( {t,0} \right)}}{{\partial {x^k}}} = \frac{{{\partial ^k}f\left( {t,1} \right)}}{{\partial {x^k}}} = 0,\,(k=0,2,4,6,8),$ $f\left( {t,x} \right) \in C_x^9\left( {{{\bar \Omega }_{\,\left( {T,1} \right)}}} \right),$  $f(\cdot,x) \in C^{1}(0,T),$ $\delta<0$, and $\psi \left( t \right)\in C^{1}\,[0,T],\,$ an increasing function that ${\psi }'\left( t \right)>0$ for all $t\in [0.T],$ then there exist a  solution $u(t,x)$ of the Problem that represented by \eqref{FS} and it has the following estimates: 
    \begin{equation*}
\left| {u(t,x)} \right| < \frac{{{C_1}}}{6} + \frac{{{C_2}}}{2}\left\| {{\omega ^{\left( 3 \right)}}\left( x \right)} \right\|_{{L_2}\left( {0,1} \right)}^2 + \frac{{{C_3}}}{2}\left\| {{\varphi ^{\left( 3 \right)}}\left( x \right)} \right\|_{{L_2}\left( {0,1} \right)}^2\,\,,
    \end{equation*}
\begin{equation*}
\left| {{u_{xx}}(t,x)} \right| < \frac{{{C_{4}}}}{6} + \frac{{{C_{5}}}}{2}\left\| {{\omega ^{\left( 5 \right)}}\left( x \right)} \right\|_{{L_2}\left( {0,1} \right)}^2 + \frac{{{C_{6}}}}{2}\left\| {{\varphi ^{\left( 5 \right)}}\left( x \right)} \right\|_{{L_2}\left( {0,1} \right)}^2\,,
\end{equation*} 
    
    \begin{equation*}
\begin{array}{l}
\left| {{}^PD_{0t;\psi }^{\alpha ,\beta  + 1,\gamma ,\delta }u(t,x)} \right| < \frac{C}{2}\left( {\left\| {{\omega ^{\left( 1 \right)}}\left( x \right)} \right\|_{{L_2}\left( {0,1} \right)}^2} \right. + \left\| {{\omega ^{\left( 3 \right)}}\left( x \right)} \right\|_{{L_2}\left( {0,1} \right)}^2 + \left\| {{\omega ^{\left( 5 \right)}}\left( x \right)} \right\|_{{L_2}\left( {0,1} \right)}^2 + \left\| {{\omega ^{\left( 7 \right)}}\left( x \right)} \right\|_{{L_2}\left( {0,1} \right)}^2\\
 + \left\| {{\varphi ^{\left( 1 \right)}}\left( x \right)} \right\|_{{L_2}\left( {0,1} \right)}^2 + \left\| {{\varphi ^{\left( 3 \right)}}\left( x \right)} \right\|_{{L_2}\left( {0,1} \right)}^2 + \left\| {{\varphi ^{\left( 5 \right)}}\left( x \right)} \right\|_{{L_2}\left( {0,1} \right)}^2 + \left\| {{\varphi ^{\left( 7 \right)}}\left( x \right)} \right\|_{{L_2}\left( {0,1} \right)}^2 + \left\| {\frac{{\partial f\left( {t,x} \right)}}{{\partial x}}} \right\|_{{L_2}\left( {0,1} \right)}^2\\
 + \int\limits_0^T {\left\| {\frac{{\partial f\left( {s,x} \right)}}{{\partial x}}} \right\|_{{L_2}\left( {0,1} \right)}^2ds}  + \int\limits_0^T {\left\| {\frac{{{\partial ^3}f\left( {s,x} \right)}}{{\partial {x^3}}}} \right\|_{{L_2}\left( {0,1} \right)}^2ds}  + \int\limits_0^T {\left\| {\frac{{{\partial ^5}f\left( {s,x} \right)}}{{\partial {x^5}}}} \right\|_{{L_2}\left( {0,1} \right)}^2ds}  + \left. {\int\limits_0^T {\left\| {\frac{{{\partial ^7}f\left( {s,x} \right)}}{{\partial {x^7}}}} \right\|_{{L_2}\left( {0,1} \right)}^2ds}  + \frac{1}{6}} \right),
\end{array}
    \end{equation*}
    
    \begin{equation*}
\begin{array}{l}
\left| {{}^PD_{0t;\psi }^{\alpha ,\beta ,\gamma ,\delta }u(t,x)} \right| < \frac{N}{2}\left( {\left\| {{\omega ^{\left( 1 \right)}}\left( x \right)} \right\|_{{L_2}\left( {0,1} \right)}^2} \right. + \left\| {{\omega ^{\left( 3 \right)}}\left( x \right)} \right\|_{{L_2}\left( {0,1} \right)}^2 + \left\| {{\omega ^{\left( 5 \right)}}\left( x \right)} \right\|_{{L_2}\left( {0,1} \right)}^2 + \left\| {{\omega ^{\left( 7 \right)}}\left( x \right)} \right\|_{{L_2}\left( {0,1} \right)}^2\\
 + \left\| {{\varphi ^{\left( 1 \right)}}\left( x \right)} \right\|_{{L_2}\left( {0,1} \right)}^2 + \left\| {{\varphi ^{\left( 3 \right)}}\left( x \right)} \right\|_{{L_2}\left( {0,1} \right)}^2 + \left\| {{\varphi ^{\left( 5 \right)}}\left( x \right)} \right\|_{{L_2}\left( {0,1} \right)}^2 + \left\| {{\varphi ^{\left( 7 \right)}}\left( x \right)} \right\|_{{L_2}\left( {0,1} \right)}^2 + \int\limits_0^T {\left\| {\frac{{\partial f\left( {t,x} \right)}}{{\partial x}}} \right\|_{{L_2}\left( {0,1} \right)}^2dt} \\
 + \int\limits_0^T {\left\| {\frac{{{\partial ^3}f\left( {t,x} \right)}}{{\partial {x^3}}}} \right\|_{{L_2}\left( {0,1} \right)}^2dt}  + \int\limits_0^T {\left\| {\frac{{{\partial ^5}f\left( {t,x} \right)}}{{\partial {x^5}}}} \right\|_{{L_2}\left( {0,1} \right)}^2dt}  + \left. {\int\limits_0^T {\left\| {\frac{{{\partial ^7}f\left( {t,x} \right)}}{{\partial {x^7}}}} \right\|_{{L_2}\left( {0,1} \right)}^2dt}  + \frac{1}{6}} \right),
\end{array}
    \end{equation*}

\begin{equation*}
    \begin{array}{l}
\left| {{}^PD_{0t;\psi }^{\alpha ,\theta ,\gamma ,\delta }{u_{xx}}(t,x)} \right| < \frac{M}{2}\left( {\left\| {{\omega ^{\left( 3 \right)}}\left( x \right)} \right\|_{{L_2}\left( {0,1} \right)}^2} \right. + \left\| {{\omega ^{\left( 5 \right)}}\left( x \right)} \right\|_{{L_2}\left( {0,1} \right)}^2 + \left\| {{\omega ^{\left( 7 \right)}}\left( x \right)} \right\|_{{L_2}\left( {0,1} \right)}^2 + \left\| {{\omega ^{\left( 9 \right)}}\left( x \right)} \right\|_{{L_2}\left( {0,1} \right)}^2\\
 + \left\| {{\varphi ^{\left( 3 \right)}}\left( x \right)} \right\|_{{L_2}\left( {0,1} \right)}^2 + \left\| {{\varphi ^{\left( 5 \right)}}\left( x \right)} \right\|_{{L_2}\left( {0,1} \right)}^2 + \left\| {{\varphi ^{\left( 7 \right)}}\left( x \right)} \right\|_{{L_2}\left( {0,1} \right)}^2 + \left\| {{\varphi ^{\left( 9 \right)}}\left( x \right)} \right\|_{{L_2}\left( {0,1} \right)}^2 + \int\limits_0^T {\left\| {\frac{{{\partial ^3}f\left( {s,x} \right)}}{{\partial {x^3}}}} \right\|_{{L_2}\left( {0,1} \right)}^2ds} \\
 + \int\limits_0^T {\left\| {\frac{{{\partial ^5}f\left( {s,x} \right)}}{{\partial {x^5}}}} \right\|_{{L_2}\left( {0,1} \right)}^2ds}  + \int\limits_0^T {\left\| {\frac{{{\partial ^7}f\left( {s,x} \right)}}{{\partial {x^7}}}} \right\|_{{L_2}\left( {0,1} \right)}^2ds}  + \left. {\int\limits_0^T {\left\| {\frac{{{\partial ^9}f\left( {s,x} \right)}}{{\partial {x^9}}}} \right\|_{{L_2}\left( {0,1} \right)}^2ds}  + \frac{1}{6}} \right).
\end{array}
\end{equation*}
\end{theorem}
Here ${\varepsilon _1},{\varepsilon _2},{\varepsilon _3},{\varepsilon _4}$ are sufficiently small, and $C,M,N$ are positive real numbers. 

Before the proof of the theorem, for convenience, we present the following lemma
\begin{lemma}
    If $\alpha_{1},\alpha_{2},\alpha_{3},\beta_{1},\beta_{2},\beta_{3},\beta_{5},\beta,\theta_{2}\in \mathbb{R}^{+},$ $\delta_{1},\delta_{2},\delta_{3},\delta_{5},\delta_{6},\delta_{7},\omega_{1},\omega_{2},\omega_{3},\omega_{4} \in \mathbb{R},$ $f(t)\in C^{1} (0, T)$, then the following formulas hold
    \begin{equation}
 \right.} \right) = \sum\limits_{m = 0}^{ + \infty } {\sum\limits_{n = 0}^{ + \infty } {\frac{{{{\left( {{\gamma _1}} \right)}_{{\alpha _1}m + {\beta _1}n}}}}{{\Gamma \left( {{\delta _1} + {\alpha _2}m + {\beta _2}n} \right)}}\frac{{{x^m}}}{{\Gamma \left( {{\delta _2} + {\alpha _3}m} \right)}}\frac{{{y^n}}}{{\Gamma \left( {{\delta _3} + {\beta _3}n} \right)}}} } 
$$
is a bivariate Mittag-Leffler-type function studied in \cite{KKh26},
$$
{}_3{{\rm E}_{13}}\left( \begin{array}{l}
{\alpha _1},{\beta _1},{\delta _1};{\beta _2},{\gamma _1},{\delta _2};\\
{\alpha _2},{\beta _3},{\gamma _2},{\delta _3};{\beta _4},{\delta _4};{\alpha _3},{\delta _5};{\beta _5},{\delta _6};{\gamma _3},{\delta _7};
\end{array} \right.\left| {\left. \begin{array}{l}
x,y\\
\,\,\,z
\end{array} \right)} \right. =
$$
$$
\sum\limits_{i = 0}^{ + \infty } {\sum\limits_{k = 0}^{ + \infty } {\sum\limits_{m = 0}^{ + \infty } {\frac{{\Gamma \left( {{\alpha _1}i + {\beta _1}k + {\delta _1}} \right)\Gamma \left( {{\beta _2}k + {\gamma _1}m + {\delta _2}} \right)\,\,\,{x^i}{y^k}{z^m}}}{{\Gamma \left( {{\alpha _2}i + {\beta _3}k + {\gamma _2}m + {\delta _3}} \right)\Gamma \left( {{\beta _4}k + {\delta _4}} \right)\Gamma \left( {{\alpha _3}i + {\delta _5}} \right)\Gamma \left( {{\beta _5}k + {\delta _6}} \right)\Gamma \left( {{\gamma _3}m + {\delta _7}} \right)}}} } }
$$
is a trivariate Mittag-Leffler-type function introduced and studied in \cite{KKh26}.

\textbf{Proof:}
Let us first prove formula~\eqref{eqv:59}. For this purpose, we rewrite 
the left-hand side of formula \eqref{eqv:59} by using 
function \eqref{SF}  as follows:
    $$
    \begin{array}{l}
{}^PI_{0t;\psi }^{{\theta _2},\beta ,\gamma ,{\omega _4}}\left( {{{\left( {{\psi _0}\left( t \right)} \right)}^{{\delta _3} - 1}}} \right. \times \\
\left. {{}_4{E_1}\left( {\begin{array}{l}
{\alpha _1},{\beta _1},{\gamma _1},{\delta _1};{\gamma _2},1,{\delta _2};\\
{\alpha _2},{\beta _2},{\gamma _3},{\theta _2},{\delta _3};{\gamma _2},{\delta _2};{\alpha _3},{\delta _5};{\beta _3},{\delta _6};{\gamma _5},{\delta _7};1,1;
\end{array}}
 \mathrel{\left | {\vphantom {\begin{array}{l}
{\alpha _1},{\beta _1},{\gamma _1},{\delta _1};{\gamma _2},1,{\delta _2};\\
{\alpha _2},{\beta _2},{\gamma _3},{\theta _2},{\delta _3};{\gamma _2},{\delta _2};{\alpha _3},{\delta _5};{\beta _3},{\delta _6};{\gamma _5},{\delta _7};1,1;
\end{array} \begin{array}{l}
{\omega _1}{\left( {{\psi _0}\left( t \right)} \right)^{{\alpha _2}}};{\omega _2}{\left( {{\psi _0}\left( t \right)} \right)^{{\beta _2}}}\\
{\omega _3}{\left( {{\psi _0}\left( t \right)} \right)^{{\gamma _3}}};{\omega _4}{\left( {{\psi _0}\left( t \right)} \right)^{{\theta _2}}}
\end{array}}}
 \right. \kern-\nulldelimiterspace}
 {\begin{array}{l}
{\omega _1}{\left( {{\psi _0}\left( t \right)} \right)^{{\alpha _2}}};{\omega _2}{\left( {{\psi _0}\left( t \right)} \right)^{{\beta _2}}}\\
{\omega _3}{\left( {{\psi _0}\left( t \right)} \right)^{{\gamma _3}}};{\omega _4}{\left( {{\psi _0}\left( t \right)} \right)^{{\theta _2}}}
\end{array}} \right) } \right)\\
 = {}^PI_{0t;\psi }^{{\theta _2},\beta ,\gamma ,{\omega _4}}\left( {{{\left( {{\psi _0}\left( t \right)} \right)}^{{\delta _3} - 1}}\sum\limits_{i = 0}^{ + \infty } {\sum\limits_{j = 0}^{ + \infty } {\sum\limits_{k = 0}^{ + \infty } {\sum\limits_{m = 0}^{ + \infty } {\frac{{\Gamma \left( {{\alpha _1}i + {\beta _1}j + {\gamma _1}k + {\delta _1}} \right){{\left( {{\gamma _2}k + {\delta _2}} \right)}_m}}}{{\Gamma \left( {{\alpha _2}i + {\beta _2}j + {\gamma _3}k + {\theta _2}m + {\delta _3}} \right)}} \times \,\,} } } } } \right.\\
\left. {\frac{{{{\left( {{\omega _1}{{\left( {{\psi _0}\left( t \right)} \right)}^{{\alpha _2}}}} \right)}^i}}}{{\Gamma \left( {{\alpha _3}i + {\delta _5}} \right)}}\frac{{{{\left( {{\omega _2}{{\left( {{\psi _0}\left( t \right)} \right)}^{{\beta _2}}}} \right)}^j}}}{{\Gamma \left( {{\beta _3}j + {\delta _6}} \right)}}\frac{{{{\left( {{\omega _3}{{\left( {{\psi _0}\left( t \right)} \right)}^{{\gamma _3}}}} \right)}^k}}}{{\Gamma \left( {{\gamma _5}k + {\delta _7}} \right)}}\frac{{{{\left( {{\omega _4}{{\left( {{\psi _0}\left( t \right)} \right)}^{{\theta _2}}}} \right)}^m}}}{{\Gamma \left( {m + 1} \right)}}} \right)\\
 = \sum\limits_{i = 0}^{ + \infty } {\sum\limits_{j = 0}^{ + \infty } {\sum\limits_{k = 0}^{ + \infty } {\frac{{\Gamma \left( {{\alpha _1}i + {\beta _1}j + {\gamma _1}k + {\delta _1}} \right){{\left( {{\omega _1}} \right)}^i}{{\left( {{\omega _2}} \right)}^j}{{\left( {{\omega _3}} \right)}^k}}}{{\Gamma \left( {{\alpha _3}i + {\delta _5}} \right)\Gamma \left( {{\beta _3}j + {\delta _6}} \right)\Gamma \left( {{\gamma _5}k + {\delta _7}} \right)}}} } } \\
 \times {}^PI_{0t;\psi }^{{\theta _2},\beta ,\gamma ,{\omega _4}}\left( {{{\left( {\psi \left( t \right) - \psi \left( 0 \right)} \right)}^{{\alpha _2}i + {\beta _2}j + {\gamma _3}k + {\delta _3} - 1}}E_{{\theta _2},{\alpha _2}i + {\beta _2}j + {\gamma _3}k + {\delta _3}}^{{\gamma _2}k + {\delta _2}}\left( {{\omega _4}{{\left( {\psi \left( t \right) - \psi \left( 0 \right)} \right)}^{{\theta _2}}}} \right)} \right).
\end{array}
    $$
Applying formula in \cite{O21}
$$
\begin{array}{l}
{}^PI_{0t;\psi }^{\alpha ,\beta ,\gamma ,\delta }\left( {{{\left( {\psi \left( t \right) - \psi \left( 0 \right)} \right)}^{\nu  - 1}}E_{\alpha ,\nu }^\sigma \left( {\delta {{\left( {\psi \left( t \right) - \psi \left( 0 \right)} \right)}^\alpha }} \right)} \right)\\
 = {\left( {\psi \left( t \right) - \psi \left( 0 \right)} \right)^{\nu  + \beta  - 1}}E_{\alpha ,\nu  + \beta }^{\sigma  + \gamma }\left( {\delta {{\left( {\psi \left( t \right) - \psi \left( 0 \right)} \right)}^\alpha }} \right)\,,
\end{array}
$$
where $\alpha,\beta,\gamma,\delta,\nu,\sigma \in\mathbb{R}$ with $\alpha  > 0,\beta > 0,\nu  > 0,$
we get
$$
\begin{array}{l}
\sum\limits_{i = 0}^{ + \infty } {\sum\limits_{j = 0}^{ + \infty } {\sum\limits_{k = 0}^{ + \infty } {\frac{{\Gamma \left( {{\alpha _1}i + {\beta _1}j + {\gamma _1}k + {\delta _1}} \right){{\left( {{\omega _1}} \right)}^i}{{\left( {{\omega _2}} \right)}^j}{{\left( {{\omega _3}} \right)}^k}}}{{\Gamma \left( {{\alpha _3}i + {\delta _5}} \right)\Gamma \left( {{\beta _3}j + {\delta _6}} \right)\Gamma \left( {{\gamma _5}k + {\delta _7}} \right)}}} } } \\
 \times {\left( {\psi \left( t \right) - \psi \left( 0 \right)} \right)^{{\alpha _2}i + {\beta _2}j + {\gamma _3}k + {\delta _3} + \beta  - 1}}E_{\alpha ,{\alpha _2}i + {\beta _2}j + {\gamma _3}k + {\delta _3} + \beta }^{{\gamma _2}k + {\delta _2} + \gamma }\left( {{\omega _4}{{\left( {\psi \left( t \right) - \psi \left( 0 \right)} \right)}^{{\theta _2}}}} \right)\\
 = {\left( {{\psi _0}\left( t \right)} \right)^{{\delta _3} + \beta  - 1}} \times \\
{}_4{E_1}\left( {\begin{array}{l}
{\alpha _1},{\beta _1},{\gamma _1},{\delta _1};{\gamma _2},1,{\delta _2} + \gamma ;\\
{\alpha _2},{\beta _2},{\gamma _3},{\theta _2},{\delta _3} + \beta ;{\gamma _2},{\delta _2} + \gamma ;{\alpha _3},{\delta _5};{\beta _3},{\delta _6};{\gamma _5},{\delta _7};1,1;
\end{array}}
 \mathrel{\left | {\vphantom {\begin{array}{l}
{\alpha _1},{\beta _1},{\gamma _1},{\delta _1};{\gamma _2},1,{\delta _2} + \gamma ;\\
{\alpha _2},{\beta _2},{\gamma _3},{\theta _2},{\delta _3} + \beta ;{\gamma _2},{\delta _2} + \gamma ;{\alpha _3},{\delta _5};{\beta _3},{\delta _6};{\gamma _5},{\delta _7};1,1;
\end{array} \begin{array}{l}
{\omega _1}{\left( {{\psi _0}\left( t \right)} \right)^{{\alpha _2}}};{\omega _2}{\left( {{\psi _0}\left( t \right)} \right)^{{\beta _2}}}\\
{\omega _3}{\left( {{\psi _0}\left( t \right)} \right)^{{\gamma _3}}};{\omega _4}{\left( {{\psi _0}\left( t \right)} \right)^{{\theta _2}}}
\end{array}}}
 \right. \kern-\nulldelimiterspace}
 {\begin{array}{l}
{\omega _1}{\left( {{\psi _0}\left( t \right)} \right)^{{\alpha _2}}};{\omega _2}{\left( {{\psi _0}\left( t \right)} \right)^{{\beta _2}}}\\
{\omega _3}{\left( {{\psi _0}\left( t \right)} \right)^{{\gamma _3}}};{\omega _4}{\left( {{\psi _0}\left( t \right)} \right)^{{\theta _2}}}
\end{array}} \right) .
\end{array}
$$
Similarly, formulas \eqref{eqv:51}-\eqref{eqv:58} and \eqref{eqv:60}-\eqref{eqv:62} can be proved by using the following formulas in \cite{O21} and \cite{O22}
$$
\begin{array}{l}
{}^PI_{0t;\psi }^{\alpha ,\beta ,\gamma ,\delta }\left( {{{\left( {\psi \left( t \right) - \psi \left( 0 \right)} \right)}^{\nu  - 1}}E_{\alpha ,\nu }^\sigma \left( {\delta {{\left( {\psi \left( t \right) - \psi \left( 0 \right)} \right)}^\alpha }} \right)} \right)\\
 = {\left( {\psi \left( t \right) - \psi \left( 0 \right)} \right)^{\nu  + \beta  - 1}}E_{\alpha ,\nu  + \beta }^{\sigma  + \gamma }\left( {\delta {{\left( {\psi \left( t \right) - \psi \left( 0 \right)} \right)}^\alpha }} \right)\,,\,\,\,\,
 \alpha  > 0,\beta > 0,\nu  > 0,
\end{array}
$$

$$
^PI_{0t;\psi }^{\alpha ,\beta ,{\kern 1pt} \gamma ,\delta }\,{\,^P}I_{0t;\psi }^{\alpha ,\nu ,{\kern 1pt} \sigma ,\delta }f(t) = {\,^P}I_{0t;\psi }^{\alpha ,\nu ,{\kern 1pt} \sigma ,\delta }\,{\,^P}I_{0t;\psi }^{\alpha ,\beta ,{\kern 1pt} \gamma ,\delta }f(t){ = ^P}I_{0t;\psi }^{\alpha ,\beta  + \nu ,{\kern 1pt} \gamma  + \sigma ,\delta }f(t),\,\,
\alpha  > 0,\beta > 0,\nu  > 0,
$$

$$
\begin{array}{l}
{}^PD_{0t;\psi }^{\alpha ,\beta ,\gamma ,\delta }\left( {{{\left( {\psi \left( t \right) - \psi \left( 0 \right)} \right)}^{\nu  - 1}}E_{\alpha ,\nu }^\sigma \left( {\delta {{\left( {\psi \left( t \right) - \psi \left( 0 \right)} \right)}^\alpha }} \right)} \right)\\
 = {\left( {\psi \left( t \right) - \psi \left( 0 \right)} \right)^{\nu  - \beta  - 1}}E_{\alpha ,\nu  - \beta }^{\sigma  - \gamma }\left( {\delta {{\left( {\psi \left( t \right) - \psi \left( 0 \right)} \right)}^\alpha }} \right)\,,\,\,\,\,
 \alpha  > 0,\beta  \ge 0,\nu  > 0,
\end{array}
$$
 
$$
{}^PD_{0t;\psi }^{\alpha ,\beta_{1} ,\gamma_{1} ,\delta }\,\,{}^PI_{0t;\psi }^{\alpha ,\beta_{2} ,\,\gamma_{2} ,\delta }f(t)={}^PI_{0t;\psi }^{\alpha ,\beta_{2}-\beta_{1} ,\,\gamma_{2}-\gamma_{1} ,\delta }f(t),\,\,\,
\alpha>0,\beta_{1}>0,\beta_{2}>0,
$$
where $\alpha,\beta_{1},\beta_{2},\gamma_{1},\gamma_{2},\delta\in\mathbb{R},$ $f\in C^{1}(\Omega)$.\\

Now we show a technique for a quadrivariate function, and this technique help us to understand our proving ideas of Theorem 2. 

Let us assume that, the expression
$$
\sum\limits_{i = 0}^{ + \infty } {\sum\limits_{j = 0}^{ + \infty } {\sum\limits_{k = 0}^{ + \infty } {{A_i}{B_j}{C_k}  } } }  
$$
be any trivariate function. 
First, we separate $k=0$ terms, then $j=0$ terms, finally $i=0$ terms. Consequently, we have
$$
\begin{array}{l}
{A_0}{B_0}{C_0} + \sum\limits_{i = 1}^{ + \infty } {{A_i}{B_0}{C_0}}  + \sum\limits_{j = 1}^{ + \infty } {{A_0}{B_j}{C_0} + } \sum\limits_{k = 1}^{ + \infty } {{A_0}{B_0}{C_k} + } \\
 + \sum\limits_{i = 1}^{ + \infty } {\sum\limits_{j = 1}^{ + \infty } {{A_i}{B_j}{C_0} + } } \sum\limits_{i = 1}^{ + \infty } {\sum\limits_{k = 1}^{ + \infty } {{A_i}{B_0}{C_k} + } } \sum\limits_{j = 1}^{ + \infty } {\sum\limits_{k = 1}^{ + \infty } {{A_0}{B_j}{C_k}} }  + \sum\limits_{i = 1}^{ + \infty } {\sum\limits_{j = 1}^{ + \infty } {\sum\limits_{k = 1}^{ + \infty } {{A_i}{B_j}{C_k}  } } }.
\end{array}
$$
In other words, we obtain the following useful equality
\begin{equation}
    \begin{array}{l}
\sum\limits_{i = 0}^{ + \infty } {\sum\limits_{j = 0}^{ + \infty } {\sum\limits_{k = 0}^{ + \infty } {{A_i}{B_j}{C_k} = } } } {A_0}{B_0}{C_0} + \sum\limits_{i = 0}^{ + \infty } {{A_{i + 1}}{B_0}{C_0}}  + \sum\limits_{j = 0}^{ + \infty } {{A_0}{B_{j + 1}}{C_0} + } \sum\limits_{k = 0}^{ + \infty } {{A_0}{B_0}{C_{k + 1}} + } \\
 + \sum\limits_{i = 0}^{ + \infty } {\sum\limits_{j = 0}^{ + \infty } {{A_{i + 1}}{B_{j + 1}}{C_0} + } } \sum\limits_{i = 0}^{ + \infty } {\sum\limits_{k = 0}^{ + \infty } {{A_{i + 1}}{B_0}{C_{k + 1}} + } } \sum\limits_{j = 0}^{ + \infty } {\sum\limits_{k = 0}^{ + \infty } {{A_0}{B_{j + 1}}{C_{k + 1}}} }  + \sum\limits_{i = 0}^{ + \infty } {\sum\limits_{j = 0}^{ + \infty } {\sum\limits_{k = 0}^{ + \infty } {{A_{i + 1}}{B_{j + 1}}{C_{k + 1}}} } } .
\end{array}
\label{UE}
\end{equation}
 We separate the function \eqref{eq:57} into three parts as follows: 
$$
%
$$
Then we rewrite each function using the Prabhakar function. After that,  we apply the equality \eqref{UE} to each  $u_{1}(t,x),u_{2}(t,x),u_{3}(t,x)$. It helps us to guarantee uniform convergence of the fractional terms of the equation \eqref{eq:31}.
$$
%
$$
Here ${\psi _0}\left( t \right) = \psi \left( t \right) - \psi \left( 0 \right),\,\,{\psi _s}\left( t \right) = \psi \left( t \right) - \psi \left( s \right),\lambda_{n}=(\pi n)^2.$

Next, we find fractional terms of the equation \eqref{eq:31} by applying formulas in \cite{O21}
$$
\begin{array}{l}
{}^PD_{0t;\psi }^{\alpha ,\beta ,\gamma ,\delta }\left( {{{\left( {\psi \left( t \right) - \psi \left( 0 \right)} \right)}^{\nu  - 1}}E_{\alpha ,\nu }^\sigma \left( {\delta {{\left( {\psi \left( t \right) - \psi \left( 0 \right)} \right)}^\alpha }} \right)} \right)\\
 = {\left( {\psi \left( t \right) - \psi \left( 0 \right)} \right)^{\nu  - \beta  - 1}}E_{\alpha ,\nu  - \beta }^{\sigma  - \gamma }\left( {\delta {{\left( {\psi \left( t \right) - \psi \left( 0 \right)} \right)}^\alpha }} \right)\,,
\end{array}
$$
where
$\alpha,\beta,\gamma,\delta,\nu,\sigma \in\mathbb{R}$ with $\alpha  > 0,\beta  \ge 0,\nu  > 0;$
$$
{}^PD_{0t;\psi }^{\alpha ,\beta ,\gamma ,\delta }\,\,{}^PI_{0t;\psi }^{\alpha ,\beta ,\,\gamma ,\delta }f(t)=f(t),
$$
with
$f \in C^{1}(\Omega)$,\, $\alpha,\beta,\gamma, \delta \in \mathbb{R}$,\, $\alpha >0$,\, $\beta>0$.
And formulas \eqref{eqv:51}-\eqref{eqv:62} to last form of functions $u_{1}(t,x),u_{2}(t,x),u_{3}(t,x)$. Then we have the following results:\\

$$
   %
} \right\rangle \\
 \times \left. {{f_n}\left( s \right)ds} \right)\left( { - {{\left( {\pi n} \right)}^2}} \right)\sin \pi nx.
\end{array}
$$
Here, ${E_{\alpha ,\beta }}\left( z \right)$ is two parametric Mittag -Leffler function, see \cite{Podlubny99} for more details.

Then, taking conditions of the theorem into account and using the estimates:
$$
\left| {E_{\alpha ,\beta }^\gamma ( - \lambda {t^\alpha })} \right| \le M,\,\,\,0<M=Const,
$$
where \,$\alpha \in \mathbb{R}^{+},\beta ,\gamma ,\lambda \in \mathbb{R}$,\, $t\in [a,b],\, a \ge 0,$ see \cite{P71} for more details;

$$
\left| {{E_{\alpha ,\beta }}\left( z \right)} \right| \le \frac{C}{{1 + \left| z \right|}},\,\,\left( {\mu  \le \left| {\arg \left( z \right)} \right| \le \pi ,\,\,\left| z \right| \ge 0\,} \right),\,
$$
where\, $\alpha<2,$ $\beta$ is an arbitrary real number, $\mu$ is such that $\frac{{\pi \alpha }}{2} < \mu  < \min \left\{ {\pi ,\pi \alpha } \right\}$ and $C$  is real constant;
$$
 \left| {{E_8}\left( \begin{array}{l}
\,\,\,\,\,\,{\gamma _1},{\alpha _1},{\beta _1};\\
{\delta _1},{\alpha _2},{\beta _2};{\delta _2},{\alpha _3};{\delta _3},{\beta _3};
\end{array} \right.\left| {\left. \begin{array}{l}
x\\
y
\end{array} \right)} \right.} \right| \le C_{1}\,,
$$
$$
\left| {{}_3{{\rm E}_{13}}\left( \begin{array}{l}
{\alpha _1},{\beta _1},{\delta _1};{\beta _2},{\gamma _1},{\delta _2};\\
{\alpha _2},{\beta _3},{\gamma _2},{\delta _3};{\beta _4},{\delta _4};{\alpha _3},{\delta _5};{1},{1};{\gamma _3},{\delta _7};
\end{array} \right.\left| {\left. \begin{array}{l}
x,y\\
\,\,\,z
\end{array} \right)} \right.} \right| \le \left( {1 + \left| y \right|} \right)\bar M,
$$
$$
\left| {{}_3{{\rm E}_{13}}\left( \begin{array}{l}
{\alpha _1},{\beta _1},{\delta _1};{\beta _2},{\gamma _1},{\delta _2};\\
{\alpha _2},{\beta _3},{\gamma _2},{\delta _3};{\beta _4},{\delta _4};{\alpha _3},{\delta _5};{1},{2};{\gamma _3},{\delta _7};
\end{array} \right.\left| {\left. \begin{array}{l}
x,y\\
\,\,\,z
\end{array} \right)} \right.} \right| \le \left( {1 + \left| y \right|} \right)\hat M,
$$
$$
\left| {{}_4{E}_1\left( {\begin{array}{l}
{\alpha _1},{\beta _1},{\gamma _1},{\delta _1};{\gamma _2},{\theta _1},{\delta _2};\\
{\alpha _2},{\beta _2},{\gamma _3},{\theta _2},{\delta _3};{\gamma _4},{\delta _4};{\alpha _3},{\delta _5};{\beta _3},{\delta _6};1,1;{\theta _3},{\delta _8};
\end{array}}
 \mathrel{\left | {\vphantom {\begin{array}{l}
{\alpha _1},{\beta _1},{\gamma _1},{\delta _1};{\gamma _2},{\theta _1},{\delta _2};\\
{\alpha _2},{\beta _2},{\gamma _3},{\theta _2},{\delta _3};{\gamma _4},{\delta _4};{\alpha _3},{\delta _5};{\beta _3},{\delta _6};1,1;{\theta _3},{\delta _8};
\end{array} \begin{array}{l}
x;y\\
z;t
\end{array}}}
 \right. \kern-\nulldelimiterspace}
 {\begin{array}{l}
x;y\\
z;t
\end{array}} \right) } \right| \le (1+\left| z \right|)\bar M_1,
$$
$$
\left| {{}_4{E}_1\left( {\begin{array}{l}
{\alpha _1},{\beta _1},{\gamma _1},{\delta _1};{\gamma _2},{\theta _1},{\delta _2};\\
{\alpha _2},{\beta _2},{\gamma _3},{\theta _2},{\delta _3};{\gamma _4},{\delta _4};{\alpha _3},{\delta _5};{\beta _3},{\delta _6};1,2;{\theta _3},{\delta _8};
\end{array}}
 \mathrel{\left | {\vphantom {\begin{array}{l}
{\alpha _1},{\beta _1},{\gamma _1},{\delta _1};{\gamma _2},{\theta _1},{\delta _2};\\
{\alpha _2},{\beta _2},{\gamma _3},{\theta _2},{\delta _3};{\gamma _4},{\delta _4};{\alpha _3},{\delta _5};{\beta _3},{\delta _6};1,2;{\theta _3},{\delta _8};
\end{array} \begin{array}{l}
x;y\\
z;t
\end{array}}}
 \right. \kern-\nulldelimiterspace}
 {\begin{array}{l}
x;y\\
z;t
\end{array}} \right) } \right| \le (1+\left| z \right|)\hat M_1
$$
Here, $C_{1},\bar M,\hat M,\bar M_1,\hat M_1$ are real constants. And other conditions of the parameters of the functions are given in \cite{KKh26}.

\newtheorem{example}{Example}
\begin{example}
    
Let us present a graphical illustration of a particular case of the solution, based on a specific choice of the given functions.
Let $\varphi \left( x \right) = \sin \pi x,\,\,\,\omega \left( x \right) = \sin \pi x,\,\,f\left( {t,x} \right) = \frac{{{{\left( {\psi \left( t \right) - \psi \left( 0 \right)} \right)}^{3\beta  + 1}}}}{{\Gamma \left( {3\beta  + 2} \right)}}\sin \pi x\,.$ Then coefficients ${A_1} = {B_1} = 1,\,\,{f_1}\left( t \right) = \frac{{{{\left( {\psi \left( t \right) - \psi \left( 0 \right)} \right)}^{3\beta  + 1}}}}{{\Gamma \left( {3\beta  + 2} \right)}}$ for $n=1$ and ${A_n} = {B_n} = {f_n}\left( t \right) = 0$ for other $n$. Consequently, the solution to the Problem obtains the following form:
$$
\begin{array}{l}
u(t,x) = \Bigg[ {{{\left( {\psi \left( t \right) - \psi \left( 0 \right)} \right)}^{\beta  - 1}}} \times \\
 {}_4{E_1}\left( {\begin{array}{l}
1,1,1,1;\gamma ,1,\gamma;\\
\beta  + 1 - \theta ,1,\beta  + 1,\alpha ,\beta ;\gamma ,\gamma ;1,1;1,1;1,1;1,1;
\end{array}}
 \mathrel{\left | {\vphantom {\begin{array}{l}
1,1,1,1;\gamma ,1,y;\\
\beta  + 1 - \theta ,1,\beta  + 1,\alpha ,\beta ;\gamma ,\gamma ;1,1;1,1;1,1;1,1;
\end{array} \begin{array}{l}
\left( { - b\,{\pi ^2}\frac{{{{\left( {\psi \left( t \right) - \psi \left( 0 \right)} \right)}^{\beta  + 1 - \theta }}}}{a}} \right);\left( { - \frac{{\left( {\psi \left( t \right) - \psi \left( 0 \right)} \right)}}{a}} \right)\\
\left( { - {\pi ^2}\frac{{{{\left( {\psi \left( t \right) - \psi \left( 0 \right)} \right)}^{\beta  + 1}}}}{a}} \right);\left( {\delta {{\left( {\psi \left( t \right) - \psi \left( 0 \right)} \right)}^\alpha }} \right)
\end{array}}}
 \right. \kern-\nulldelimiterspace}
 {\begin{array}{l}
\left( { - b\,{\pi ^2}\frac{{{{\left( {\psi \left( t \right) - \psi \left( 0 \right)} \right)}^{\beta  + 1 - \theta }}}}{a}} \right);\left( { - \frac{{\left( {\psi \left( t \right) - \psi \left( 0 \right)} \right)}}{a}} \right)\\
\left( { - {\pi ^2}\frac{{{{\left( {\psi \left( t \right) - \psi \left( 0 \right)} \right)}^{\beta  + 1}}}}{a}} \right);\left( {\delta {{\left( {\psi \left( t \right) - \psi \left( 0 \right)} \right)}^\alpha }} \right)
\end{array}} \right)+ \\
  \frac{{\left( {a + 1} \right)}}{a}{\left( {\psi \left( t \right) - \psi \left( 0 \right)} \right)^\beta }\times\\
 {}_4{E_1}\left( {\begin{array}{l}
1,1,1,1;\gamma ,1,\gamma;\\
\beta  + 1 - \theta ,1,\beta  + 1,\alpha ,\beta  + 1;\gamma ,\gamma ;1,1;1,1;1,1;1,1;
\end{array}}
 \mathrel{\left | {\vphantom {\begin{array}{l}
1,1,1,1;\gamma ,1,y;\\
\beta  + 1 - \theta ,1,\beta  + 1,\alpha ,\beta  + 1;\gamma ,\gamma ;1,1;1,1;1,1;1,1;
\end{array} \begin{array}{l}
\left( { - b\,{\pi ^2}\frac{{{{\left( {\psi \left( t \right) - \psi \left( 0 \right)} \right)}^{\beta  + 1 - \theta }}}}{a}} \right);\left( { - \frac{{\left( {\psi \left( t \right) - \psi \left( 0 \right)} \right)}}{a}} \right)\\
\left( { - {\pi ^2}\frac{{{{\left( {\psi \left( t \right) - \psi \left( 0 \right)} \right)}^{\beta  + 1}}}}{a}} \right);\left( {\delta {{\left( {\psi \left( t \right) - \psi \left( 0 \right)} \right)}^\alpha }} \right)
\end{array}}}
 \right. \kern-\nulldelimiterspace}
 {\begin{array}{l}
\left( { - b\,{\pi ^2}\frac{{{{\left( {\psi \left( t \right) - \psi \left( 0 \right)} \right)}^{\beta  + 1 - \theta }}}}{a}} \right);\left( { - \frac{{\left( {\psi \left( t \right) - \psi \left( 0 \right)} \right)}}{a}} \right)\\
\left( { - {\pi ^2}\frac{{{{\left( {\psi \left( t \right) - \psi \left( 0 \right)} \right)}^{\beta  + 1}}}}{a}} \right);\left( {\delta {{\left( {\psi \left( t \right) - \psi \left( 0 \right)} \right)}^\alpha }} \right)
\end{array}} \right)+ \\
  \frac{1}{a}{\left( {\psi \left( s \right) - \psi \left( 0 \right)} \right)^{4\beta  + 2}}\times\\
\left. {  {}_4{E_1}\left( {\begin{array}{l}
1,1,1,1;\gamma ,1,\gamma;\\
\beta  + 1 - \theta ,1,\beta  + 1,\alpha ,4\beta  + 3;\gamma ,\gamma ;1,1;1,1;1,1;1,1;
\end{array}}
 \mathrel{\left | {\vphantom {\begin{array}{l}
1,1,1,1;\gamma ,1,y;\\
\beta  + 1 - \theta ,1,\beta  + 1,\alpha ,4\beta  + 3;\gamma ,\gamma ;1,1;1,1;1,1;1,1;
\end{array} \begin{array}{l}
\left( { - b\,{\pi ^2}\frac{{{{\left( {\psi \left( t \right) - \psi \left( 0 \right)} \right)}^{\beta  + 1 - \theta }}}}{a}} \right);\left( { - \frac{{\left( {\psi \left( t \right) - \psi \left( 0 \right)} \right)}}{a}} \right)\\
\left( { - \,{\pi ^2}\frac{{{{\left( {\psi \left( t \right) - \psi \left( 0 \right)} \right)}^{\beta  + 1}}}}{a}} \right);\left( {\delta {{\left( {\psi \left( t \right) - \psi \left( 0 \right)} \right)}^\alpha }} \right)
\end{array}}}
 \right. \kern-\nulldelimiterspace}
 {\begin{array}{l}
\left( { - b\,{\pi ^2}\frac{{{{\left( {\psi \left( t \right) - \psi \left( 0 \right)} \right)}^{\beta  + 1 - \theta }}}}{a}} \right);\left( { - \frac{{\left( {\psi \left( t \right) - \psi \left( 0 \right)} \right)}}{a}} \right)\\
\left( { - \,{\pi ^2}\frac{{{{\left( {\psi \left( t \right) - \psi \left( 0 \right)} \right)}^{\beta  + 1}}}}{a}} \right);\left( {\delta {{\left( {\psi \left( t \right) - \psi \left( 0 \right)} \right)}^\alpha }} \right)
\end{array}} \right) } \right]\times\\
 \sin \pi x\,.
\end{array}
$$
\end{example}

\bibliographystyle{cas-model2-names}

\bibliography{cas-refs}



\end{document}